\documentclass[10pt]{amsart}
\usepackage{amsmath, amssymb, amscd, amsthm, mathrsfs, url ,pinlabel, verbatim, lipsum, wrapfig}
\usepackage[text={6.5in,9.5in}, centering, letterpaper, dvips]{geometry}
\usepackage{color, multirow, dcpic, latexsym, pictexwd, graphicx, epstopdf, tikz-cd, wrapfig, float, xcolor, cancel, subfiles}
\usepackage{verbatim, relsize, latexsym, latexsym, tikz, mdframed, float, mathtools, flowchart, multicol, stmaryrd, caption, capt-of, old-arrows, dirtytalk, footmisc, setspace, enotez, tabularx, pifont, tensor, csquotes}
\usepackage{svg}
\usepackage{fancyhdr, tikz, tikz-cd, setspace, tensor, enumitem, dirtytalk, nicematrix}

\usetikzlibrary{positioning, arrows.meta, knots, braids}

\usepackage{scalerel,stackengine}
\stackMath
\newcommand\reallywidehat[1]{%
\savestack{\tmpbox}{\stretchto{%
  \scaleto{%
    \scalerel*[\widthof{\ensuremath{#1}}]{\kern-.6pt\bigwedge\kern-.6pt}%
    {\rule[-\textheight/2]{1ex}{\textheight}}
  }{\textheight}%
}{0.5ex}}%
\stackon[1pt]{#1}{\tmpbox}%
}

\usepackage[backref=page, linktocpage=true]{hyperref}

\renewcommand*{\backref}[1]{}
\renewcommand*{\backrefalt}[4]{%
    \ifcase #1 (Not cited.)%
    \or        (Cited on page~#2.)%
    \else      (Cited on pages~#2.)%
    \fi}

\definecolor{maroon}{rgb}{0.5, 0.0, 0.0}
\definecolor{darkblue}{rgb}{0.03, 0.27, 0.49}

\hypersetup{
  colorlinks   = true, 
  urlcolor     = maroon, 
  linkcolor    = darkblue, 
  citecolor   = maroon 
}

\usepackage{svg}
\usepackage{subfig}

\usepackage[T1]{fontenc}
\usepackage{lmodern}
\DeclareUrlCommand{\bfurl}{}
\DeclareMathOperator{\Fix}{Fix}
\DeclareMathOperator{\Kaw}{Kaw}
\DeclareMathOperator{\Diff}{Diff}

\DeclareMathOperator{\id}{id}
\DeclareMathOperator{\lk}{lk}

\DeclareMathOperator{\Or}{O}
\DeclareMathOperator{\Wh}{Wh}

\newtheorem{thm}{Theorem}[section]

\newtheorem{lem}[thm]{Lemma}
\newtheorem{cor}[thm]{Corollary}
\newtheorem{prop}[thm]{Proposition}

\theoremstyle{definition}
\newtheorem{defn}[thm]{Definition}

\newtheorem{exmp}[thm]{Example}
\newtheorem{remark}[thm]{Remark}
\newtheorem{question*}{Question}

\theoremstyle{plain}
\newtheorem{introthm}{Theorem}
\newtheorem{introcor}[introthm]{Corollary}

\renewcommand{\int}{\operatorname{int}
}

\newtheorem*{rep@introthm}{\rep@title}
\newcommand{\newrepintrothm}[2]{
\newenvironment{rep#1}[1]{
  \def\rep@title{#2 \ref{##1}}
 \begin{rep@introthm}}
 {\end{rep@introthm}}} 

 \theoremstyle{definition}
 \newrepintrothm{introthm}{Theorem}

\let\epsilon\varepsilon
\theoremstyle{remark}

\newcommand\Z{\mathbb{Z}}
\newcommand\Q{\mathbb{Q}}
\newcommand\R{\mathbb{R}}
\newcommand\C{\mathcal{C}}

\newcommand\AC{\mathcal{AC}}
\newcommand\G{\mathbb{G}}

\usepackage[colorinlistoftodos,textwidth=3cm]{todonotes}

\newlength\Colsep
\title{Every negative amphichiral knot is rationally slice}

\author{Alessio Di Prisa}
\address{Scuola Normale Superiore, 56126 Pisa, Italy}
\email{\url{alessio.diprisa@sns.it}}
\urladdr{\url{https://sites.google.com/view/alessiodiprisa}}

\author{Jaewon Lee}
\address{Department of Mathematical Sciences, KAIST, 34141 Daejeon, Republic of Korea}
\email{\url{freejw@kaist.ac.kr}}
\urladdr{\url{https://mathsci.kaist.ac.kr/~freejw}}

\author{O{\u{g}}uz \c{S}avk}
\address{CNRS and Laboratorie de Math\'ematiques Jean Leray, Nantes Universit\'e, 44322 Nantes, France}
\email{\url{oguz.savk@cnrs.fr}}
\urladdr{\url{https://sites.google.com/view/oguzsavk}}

\date{}

\begin{document}

\begin{abstract}
In 2009, Kawauchi proved that every strongly negative amphichiral knot is rationally slice. However, as shown by Hartley in 1980, there are examples of negative amphichiral knots that are not strongly negative amphichiral. In this paper, we prove that every negative amphichiral link whose amphichiral map preserves each component is rationally slice. Our proof relies on a systematic analysis of the action induced by the negative amphichiral map on the JSJ decomposition of the link exterior. Moreover, we provide sufficient conditions on such an action to deduce when a negative amphichiral knot is either isotopic to, or concordant to, a strongly negative amphichiral knot. In particular, we prove that every fibered negative amphichiral knot is strongly negative amphichiral, answering a question asked by Kim and Wu in 2016 on Miyazaki knots.
\end{abstract}

\maketitle

\section{Introduction}
\label{sec:intro}

The notions of sliceness and concordance for knots were introduced in the seminal article by Fox and Milnor \cite{FM66}. Since then, these concepts have played an important role in the development of low-dimensional topology. See the surveys \cite{Liv05, Hom17, Hom23, Sav24} for a comprehensive discussion.

Certain types of knot symmetries provide examples of interesting phenomena in knot concordance. For instance, every \emph{negative amphichiral knot}, i.e., a knot $K$ invariant under an orientation-reversing diffeomorphism $f$ of $S^3$ mapping the knot $K$ to itself with the opposite orientation, has concordance order at most $2$. If such a map $f$ is chosen to be an involution, we say that $K$ is a \emph{strongly negative amphichiral knot}. The figure-eight knot $4_1$ is the simplest strongly negative amphichiral knot, and hence the simplest non-slice knot of order $2$.

Cochran, based on \cite{FS84}, observed that $4_1$ bounds a smooth disk in a rational ball, i.e., it is \textit{rationally slice}. Kawauchi \cite{Kaw80} also proved that the $(2,1)$-cable of $4_1$ is rationally slice as well. Cha \cite{Cha07} generalized Cochran's argument to prove that infinitely many non-concordant strongly negative amphichiral knots are rationally slice. Later, Kawauchi \cite{Kaw09} eventually showed that every strongly negative amphichiral knot is rationally slice.

While the rational ball constructed by Kawauchi \cite{Kaw09} a priori depends on the choice of a strongly negative amphichiral knot, Levine \cite{Lev23} proved that every such rational ball is diffeomorphic to the same manifold $V$, called \textit{Kawauchi manifold}.\footnote{The closed rational homology $4$-sphere obtained from $V$ by gluing a $4$-ball is often called \textit{Kawauchi manifold} as well. For a closed $4$-manifold $X$, a knot in $S^3$ is also called \textit{slice in $X$} if it bounds a smooth disk in a punctured $X$. These terms are used interchangeably, see \cite{Lev23}.}

On the other hand, it was previously shown by Hartley \cite{Har80} that there exist negative amphichiral knots that are not strongly negative amphichiral. Since the existence of an involution plays a crucial role in both constructions \cite{Kaw09, Lev23}, it is natural to ask whether it is strictly necessary, or alternatively, if any negative amphichiral knot is rationally slice. While Hartley's construction produces many rationally slice knots, it is not immediate to give a general answer to this question because of the large number of possible choices of satellite patterns, which makes the situation rather arbitrary and hard to control. In this paper, we answer this question more generally, for \emph{negative amphichiral links}, i.e., links with an orientation-reversing diffeomorphism of $S^3$ which sends each component of the link to itself with the reversed orientation.

\begin{introthm}\label{thm:link}
    Every negative amphichiral link is rationally slice.
\end{introthm}

In particular, every negative amphichiral knot is rationally slice. In fact, the inductive part in the proof of Theorem \ref{thm:link}, combined with \cite{Lev23}, implies that every negative amphichiral link is slice in the boundary connected sum of $n$ copies of the Kawauchi manifold $V$ for some $n\ge 0$.

For a link $L$, define $\Kaw(L)$ as the minimal number among such $n$, or $\infty$ if there is no such $n$. Then it is clear that $\Kaw(L)$ is a concordance invariant. When $L$ is a slice knot $K$, $\Kaw(K)=0$. Every strongly negative amphichiral knot $K$ has $\Kaw(K)\le 1$ \cite{Kaw09,Lev23}. For any negative amphichiral knot $K$, Theorem~\ref{thm:link} implies that $\Kaw(K) <\infty$. One can ask the following problem for which we will provide potential candidates, see Section~\ref{sec:weak} for more details.

\begin{question*}\label{question:single-Kaw}
Does there exist a negative amphichiral knot $K$ with $\Kaw(K) > 1$?
\end{question*}

Recall that every non-trivial knot is known to be either a torus knot, a hyperbolic knot, or a satellite knot by Thurston \cite{Thu97}. Any torus knot is obviously not negative amphichiral. Kawauchi \cite{Kaw79} proved that every hyperbolic negative amphichiral knot is strongly negative amphichiral, and hence it is rationally slice by \cite{Kaw09}. Thus, the remaining case of negative amphichiral knots, for which it was not completely known whether they are rationally slice is given by satellite negative amphichiral knots.

More recently, Kim and Wu \cite{KW18} proved that every \emph{Miyazaki knot}, i.e., a fibered negative amphichiral knot with irreducible Alexander polynomial, is rationally slice. More precisely, given a satellite Miyazaki knot $P(K)$, they proved that the pattern $P$ is unknotted, and the companion $K$ is again a Miyazaki knot. Based on induction on $3$-genus, they conclude that such a knot is obtained by iterating satellite operations with unknotted patterns on a hyperbolic Miyazaki knot.

In order to avoid any additional assumption, such as fiberedness and irreducibility of the Alexander polynomial, we consider a generalized satellite operation, known as \textit{splice} of links. A key ingredient in the paper is the \textit{companionship graph} of a link, extensively studied by Budney \cite{Bud06}, which is a labeled and partially directed tree derived from the \textit{JSJ decomposition} \cite{JS78, Joh79} of the link exterior. This graph encodes the information on how the link can be obtained via iterated splicing of other links.

Moreover, the amphichiral map of an amphichiral link induces an action on its companionship graph. Analyzing this action allows us to determine how a given such a link can arise from the splicing of simpler amphichiral links. By introducing a notion of \textit{complexity of links} defined from their companionship graphs, such analysis allows us to proceed by induction in the proof of Theorem~\ref{thm:link}.

\subsection{Fibered negative amphichiral knots}
Recall that every hyperbolic negative amphichiral knot is strongly negative amphichiral \cite{Kaw79}. Following this idea, it is natural to ask whether it is possible to give different sufficient conditions for a negative amphichiral knot to be strongly negative amphichiral. For example, Kim and Wu \cite{KW18} asked if every Miyazaki knot is strongly negative amphichiral.

In the proof of Theorem~\ref{thm:link}, it turns out that if the companionship graph of a negative amphichiral link has no edges fixed by the amphichiral action, then it is strongly negative amphichiral. In particular, since the exterior of a hyperbolic knot has a single JSJ piece, the result by \cite{Kaw79} falls in this case. Motivated by this fact, we define a notion of \textit{totally coherent} JSJ structure with respect to the amphichiral action, which generalizes the cases where there are no fixed edges, and we prove the following.

\begin{introthm}\label{thm:SNACK}
    If a negative amphichiral knot admits a totally coherent JSJ structure, then it is strongly negative amphichiral.    
\end{introthm}

We briefly explain what totally coherent JSJ structure means. Suppose a satellite knot $P(K)\subset S^3$ is negative amphichiral. Regard $S^3$ as the union of a solid torus containing the pattern $P$ and the exterior of $K$ along the torus boundary. When the torus is fixed by the amphichiral map, the companion knot $K$ is either negative or positive amphichiral. Roughly speaking, we call the former situation \textit{coherent} and the latter one \textit{incoherent}, in the general splice setting. When every such a fixed torus in the JSJ decomposition of a negative amphichiral knot exterior is coherent, we say that the knot has a \textit{totally coherent} JSJ structure.

It turns out that every \textit{fibered} negative amphichiral knot indeed falls under the totally coherent condition. Consequently, we obtain the following corollary, which, in particular, answers the question asked in \cite{KW18} on Miyazaki knots.

\begin{introcor}\label{cor:fibered}
Every fibered negative amphichiral knot is strongly negative amphichiral.
\end{introcor}

Miyazaki \cite{Miy94} proved that the $(2n, 1)$-cable of every Miyazaki knot with $n > 0$ is not ribbon. Kim and Wu \cite{KW18} extended this result to all non-trivial linear combinations of such cables. While a few of them are known to be not slice \cite{DKMPS24, ACMPS23, KPT24, KPT25}, many of them remain as a potential counterexample to the slice-ribbon conjecture. One of the difficulties in obstructing such knots being slice comes from the fact that they are \textit{strongly} rationally slice, i.e., the inclusion from the knot complement to the slice disk complement induces an isomorphism on the free part of first homology in integer coefficients. In this context, one can ask a general question:

\begin{question*}
  Is the $(2n, 1)$-cable of every non-slice fibered negative amphichiral knot not ribbon?
  \label{question:ribbon}
\end{question*}

\subsection{Weakly negative amphichiral knots}
\label{sec:weak}

While Theorem~\ref{thm:link} immediately implies that every negative amphichiral knot is rationally concordant to any other strongly negative amphichiral knot, it is not known if such a concordance exists in $S^3\times [0, 1]$. More precisely, one might ask:
\begin{question*}\label{question:con-to-SNACK}
  Is every negative amphichiral knot concordant to a strongly negative amphichiral knot?
\end{question*}

If such a concordance always exists, then the answer to Question~\ref{question:single-Kaw} is automatically negative. In this direction, we find a sufficient condition on the companionship graph for a negative amphichiral knot to be concordant to a strongly negative amphichiral knot. We call such a condition \textit{properly incoherent}, and prove:

\begin{introthm}\label{thm:SNACK-concordance}
    If a negative amphichiral knot admits a properly incoherent JSJ structure, then it is concordant to a strongly negative amphichiral knot.
\end{introthm}

We call a negative amphichiral knot \textit{weakly negative amphichiral} if it is not strongly negative amphichiral. Since the properly incoherent condition is quite restrictive, we can construct many weakly negative amphichiral knots which do not admit properly incoherent JSJ structures in Propositions~\ref{prop:cand1} and \ref{prop:cand2}. For the resulting knots in Proposition~\ref{prop:cand1}, we can check that at least they are not slice, while the ones in Proposition~\ref{prop:cand2} are \textit{topologically} slice. We do not know whether they are smoothly concordant to some strongly negative amphichiral knot.

We note that if the answer to Question~\ref{question:con-to-SNACK} is negative in general, then it would be difficult to find an obstruction, since most concordance properties of strongly negative amphichiral knots are also shared by negative amphichiral knots (for example, see \cite[Theorem~3.1]{Har80b}). A possible approach in this direction would be to find an example for Question~\ref{question:single-Kaw}. We remark that the knots constructed in Propositions~\ref{prop:cand1} and \ref{prop:cand2} may be potential examples for Question~\ref{question:single-Kaw} and hence, potential counterexamples to Question~\ref{question:con-to-SNACK}.

\subsection{Remark on $2$-torsions in the knot concordance group}
Knot concordance classes form an abelian group $\C$, called the \emph{knot concordance group}, with the operation induced by connected sum, and inverse given by taking orientation-reversed mirror. By definition, every non-slice negative amphichiral knot is of order $2$ in $\C$. Gordon asked whether the converse holds \cite[Problem~16]{Gor78}, namely if every knot of order $2$ in $\C$ is concordant to a negative amphichiral knot. See also \cite[p. 266]{FM66} and \cite[Problem~1.94]{Kir97}. It is also not known if there exists any non-slice knot of finite order in $\C$ other than $2$ \cite[Problem~1.32]{Kir97}. In this subsection, we make remarks on several related questions.

The \textit{rational knot concordance group} $\C_\Q$ is defined as the quotient group of $\C$ modulo rationally slice knots. The difference between $\C$ and $\C_\Q$ has been studied extensively \cite{Cha07, Kaw09, CHL11, BD12, HKL16, AMMMPS22, Col22, Mil22, ACMPS23, DKMPS24, KPT24, Liv24, HKP25, KPT25}, but $\C_\Q$ is also of independent interest \cite{CO93, CK02, Cha07, CFHH13, KW18, Lev23, Kim23, CHKPW24, Lee24, Lee25}. Since every known knot of order $2$ in $\C$ is rationally slice by Theorem~\ref{thm:link}, an analogous question in $\C_\Q$ would arise as follows:

\begin{question*}\cite[Question~2]{Lee24}
\label{question:2-torsion-CQ}
  Does there exist a $2$-torsion element in $\C_\Q$?
\end{question*}

As the diagram below shows, knots of order $2$ in the \textit{algebraic rational concordance group} $\AC_\Q$, which is defined by Cha \cite{Cha07}, are potential candidates for Question~\ref{question:2-torsion-CQ}. However, it was proved that infinitely many such knots are linearly independent in $\C_\Q$ in \cite{Lee24}.
\begin{center}
\begin{tikzcd}
&\C \ar[twoheadrightarrow]{r}\ar[twoheadrightarrow]{d} &\C_\Q \ar[twoheadrightarrow]{d}\\
&\AC \ar[twoheadrightarrow]{r} &\AC_\Q
\end{tikzcd}
\end{center}
On the other hand, weakly negative amphichiral knots could have been regarded as more likely candidates for Question~\ref{question:2-torsion-CQ} since they have order at most $2$ in $\C$, and hence, in $\C_\Q$. Note that from Theorem~\ref{thm:link} we are able to exclude all such candidates.

We remark that Question~\ref{question:2-torsion-CQ} still remains unresolved. If the answer to Gordon's question turns out to be negative, i.e., if there exists a knot of order $2$ in $\C$ which is not concordant to any negative amphichiral knot, then this would also be a potential candidate for a $2$-torsion element in $\C_\Q$.

Gordon's question and Question~\ref{question:2-torsion-CQ} are also related to the geography and botany problem of knot Floer complexes \cite{OS04, Ras03}. See for instance \cite{HW18, Pop23}. Let $\mathfrak{K}$ be the \textit{local equivalence group} of knot Floer complexes over $\mathbb{F}[U, V]$ without the involution defined in \cite{Zem19}.\footnote{The notion of local equivalence of knot Floer complexes over $\mathbb{F}[U, V]$ is equivalent with the \textit{stable equivalence} in \cite{Hom17} or the \textit{$\nu^+$-equivalence} in \cite{KP18}. A similar notion for knot Floer complexes over $\mathbb{F}[U, V]/UV$ was previously defined in \cite{Hom14} as the \textit{$\varepsilon$-equivalence.}} Recall that concordance invariants such as $\tau$-invariant \cite{OS03}, $\nu$-invariant \cite{OS11}, $\epsilon$-invariant \cite{Hom14}, $\nu^+$-invariant \cite{HW16}, and $\Upsilon$-invariant \cite{OSS17} from the knot Floer complex without the involution are determined by the local equivalence class of $[CFK(K)]\in \mathfrak{K}$. While there exists a knot Floer complex $C \in \mathfrak{K}$ described in Figure~\ref{fig:CFK} such that $C$ is of order $2$ in $\mathfrak{K}$ but $\nu^+(C) \neq 0$ \cite[Remark~2.1]{Hom15}, it is not known if there exists a knot $K$ such that $[CFK(K)] = C$.

\begin{figure}[h]
  \includesvg{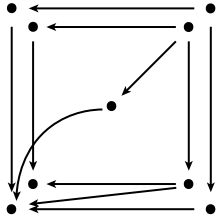}
  \caption{A knot Floer complex $C$ of order $2$ with $\nu^+(C) = 1$.}
\label{fig:CFK}
\end{figure}

While $\tau(K)$, $\nu(K)$, $\epsilon(K)$, and $\Upsilon_K(t)$ vanish for any knot $K$ of finite concordance order, it is not known whether $\nu^+$-invariant or its generalization $\nu_n$-invariant \cite{Tru19} vanish for such $K$. On the other hand, the local equivalence class $[CFK(K)]$ is trivial if $K$ is rationally slice \cite{Hom17} (see also \cite[p. 2]{HKPS22}). See the diagram below.

\begin{center}
\begin{tikzcd}
&\C \ar{r}\ar[twoheadrightarrow]{d} &\mathfrak{K}\\
&\C_\Q \ar[dashed]{ur}
\end{tikzcd}
\end{center}
Thus, we obtain the following immediate corollary from Theorem~\ref{thm:link}:
\begin{introcor}
    If a knot $K$ is negative amphichiral, then the image $[CFK(K)]$ in $\mathfrak{K}$ is trivial.
\end{introcor}

\noindent In particular, for any negative amphichiral knot $K$, we have $$\nu^+(K) = \nu^+(\overline{K}) = \nu_n(K)=\nu_n(\overline{K})=0 \text{ for any integer }n,$$

\noindent which answers \cite[Question~3]{KW18} affirmatively.

Since knots concordant to a negative amphichiral knot are the only known knots of finite concordance order in $\C$, it still remains open whether there exists a knot of non-trivial finite order in $\mathfrak{K}$. Since the homomorphism $\mathcal{C}\rightarrow\mathfrak{K}$ factors through $\C_\Q$, if the answer to Question~\ref{question:2-torsion-CQ} is positive, it would give an interesting candidate for such a knot.

Note that it was proved that not all rationally slice knots are of finite order in $\C$ \cite{HKPS22}. However, Theorem~\ref{thm:link} implies that every known knot of finite order in $\C$ is rationally slice. We close the introduction with the following general question, also highlighted by Kim and Wu \cite[Question~2]{KW18}.

\begin{question*}
  If a knot $K$ has finite order in $\C$, then is $K$ rationally slice?
\end{question*}

\subsection*{Notations} 
Let $L$ be a link in $S^3$ indexed by a finite set $A$.
\begin{itemize}
    \item For a subset $A'\subset A$, $L_{A'}$ denotes the sublink of $L$ indexed by $A'$.
    \item For an element $a\in A$, $L_a$ denotes the component of $L$ indexed by $a$.
    \item $|L|$ denotes the number of the components of $L$.
    \item $\overline{L}$ denotes the mirror image of $L$.
    \item For an oriented knot $K$, $\pm K$ denotes the same knot with the same or the opposite orientation, respectively.
    \item For a smoothly embedded oriented submanifold $Y$ in $X$, the tubular neighborhood of $Y$ in $X$ is denoted by $\nu Y$. The complement $X-Y$ denotes the closure of $X-\nu Y$.
    \item The link exterior $S^3-L$ is denoted by $E_L$ unless otherwise stated.
    \item The index set $A$ of the link $L$ is occasionally regarded as the set of the boundary components of $E_L$, i.e., $$A=\{\partial \nu L_a\ |\ a\in A\}.$$
    \item We use the shorthand notation $RHX$ for a manifold $Y$ of the same dimension with $X$ such that $H_*(Y;R) \cong H_*(X;R)$.
    \item $U^n$ denotes an $n$-component unlink unless otherwise stated.
    \item We denote the interior of $X$ by $\int(X)$.
\end{itemize}

\subsection*{Organization} 

In Section~\ref{sec:2}, we provide preliminaries on JSJ decompositions and companionship graphs, Gromov norms, and the Smith theory. We also introduce the notion of the complexity of links and prove basic properties therein. In Section~\ref{sec:3}, we study the actions of the amphichiral maps on the companionship graphs and prove several essential properties. We also define totally coherent and properly incoherent JSJ structure of negative amphichiral knot which appear in Theorems~\ref{thm:SNACK} and \ref{thm:SNACK-concordance}. In Section~\ref{sec:4}, we consider concordance of spliced links in a general $4$-manifold and prove Theorem~\ref{thm:link}. In Section~\ref{sec:5}, we prove Theorems~\ref{thm:SNACK} and ~\ref{thm:SNACK-concordance}, and also Corollary~\ref{cor:fibered}. In Section~\ref{sec:6}, we present potential candidates for Question~\ref{question:single-Kaw}: weakly negative amphichiral knots that may not be slice in the Kawauchi manifold $V$.

\subsection*{Acknowledgments} 

This project was initiated during the 2024 session of the Institut Fourier Summer School in Grenoble, and we would like to thank the organizers of the event. We appreciate Jae Choon Cha and JungHwan Park for valuable comments. The second author is partially supported by the Samsung Science
and Technology Foundation (SSTF-BA2102-02) and the NRF grant RS-2025-00542968.


\section{Preliminaries}
\label{sec:2}

In this section, we give several preliminaries. We first recall the JSJ decompositions of link exteriors and its associated companionship graphs. Next, we introduce the notion of complexity of links, depending on Gromov norms of the exterior and the number of vertices in companionship graphs. We then prove several properties of complexity of links, including monotonicity and well-ordering. We finally recall several results on the Smith theory and prove a lemma relating the cardinality of the orbits of the components of a certain hyperbolic link under a rotation. 

\subsection{Companionship graphs of JSJ decompositions and splicing operation for links}
\label{sec:jsj}
In this subsection, we mainly follow Budney's survey article \cite{Bud06}. One can also consult the manuscript by Bonahon and Siebenmann \cite{BS79} for fundamental results and basic notions.

Let $L \subset S^3$ be a non-split link with the exterior $E_L=S^3 - L$. Then, up to isotopy, the \emph{JSJ decomposition of} $E_L$ is given by a unique minimal set of disjoint incompressible tori $\mathcal{T}=\{T_1,\ldots, T_n\}$ such that
\begin{itemize}
    \item each component of $E_L - \cup_i T_i$ is either Seifert fibered or hyperbolic,
    \item each $T_i$ is not boundary parallel,
    \item each pair $T_i$ and $T_j$ are not parallel.
\end{itemize} 

We first introduce the JSJ graph that records the information of JSJ decomposition of a link exterior. We call each component of $E_L- \cup_i \ T_i$ a \textit{JSJ piece} and each torus $T_i\in \mathcal{T}$ a \textit{JSJ torus}.

\begin{defn}\cite[Definition 5]{Bud06}
\label{defn:JSJ graph}
For a non-split link $L$ in $S^3$ indexed by $A$, the \emph{JSJ graph $G_L$ of $L$} is a partially directed graph defined as follows: 
\begin{itemize}
    \item the vertex set is the set of the JSJ pieces of $E_L$,
    \item the edge set is $\mathcal{T}$, the set of the JSJ tori of $E_L$,
    \item each $T\in \mathcal{T}$ joins vertices containing $T$,
    \item The edge $T$ joining $v$ and $w$ is directed from $v$ to $w$ if exactly one of the components of $S^3-T$ is a solid torus $Y$, and $w\subset Y$.
\end{itemize}
For a split link $L$ indexed by $A= \sqcup_{i=1}^k A_i$ so that each sublink $L_{A_i}$ is non-split, $E_L = \#_{i=1}^k E_{L_{A_i}}$ gives the prime decomposition of $E_L$. In this case define the JSJ graph $G_L$ of $L$ as $\sqcup_{i=1}^k G_{L_{A_i}}$.
\end{defn}

From now on, by abuse of notation, we use the same symbol, such as $v$, to denote both a JSJ piece and a vertex. Similarly, the symbol $T$ will denote both a JSJ torus and an edge. Notice that since every torus in $S^3$ is separating, for a non-split link $L$, the graph $G_L$ is always a tree. Since any such a torus bounds at least one solid torus, the undirected edge bounds solid tori in both sides.

For a $3$-manifold $M\subset S^3$ whose boundary is a disjoint union of tori, let $T$ be a component of $\partial M$, and let $C$ be a component of $S^3 - M$ containing $T$. An essential curve $c \subset T$ is called a \emph{peripheral curve for $M$} if there exists a properly embedded surface $S$ in $C$ such that $c = \partial S$. 

\begin{lem} \cite[Proposition 3]{Bud06}\label{lemma:untwisted re-embedding}
Let $M$ be a $3$-manifold in $S^3$ with boundary tori. Then, up to isotopy, there exists a unique orientation-preserving embedding $f:M\to S^3$ such that
\begin{itemize}
    \item $f(M) = E_L$ for some link $L$ in $S^3$,
    \item $f$ maps peripheral curves of $\partial M$ to peripheral curves of $\partial (f(M))$.
\end{itemize}
\end{lem}
\noindent Such an embedding is said to be the \emph{untwisted re-embedding} of $M$. We will denote the link corresponding to the untwisted re-embedding of $M$ by $L(M)\subset S^3$.

\begin{defn}\cite[Definition~6]{Bud06}
\label{defn:companionship-graph}
For a non-split link $L$ in $S^3$ indexed by $A$, let $\mathcal{T}$ be the JSJ tori of $E_L$. We define the \emph{companionship graph $\G_L$ of $L$} as follows:
\begin{itemize}
    \item The underlying partially directed graph without its vertex labels is the JSJ graph $G_L$ of $L$.

    \item Each vertex $v$ is labeled by the corresponding link $L(v)$, which is obtained from Lemma~\ref{lemma:untwisted re-embedding}.

\end{itemize}

\noindent The link $L(v)$ is indexed by the disjoint union of
\begin{itemize}
  \item the subset $A(v)$ of $A$ corresponding to the components of $v \cap \partial E_L$: $$A(v)=\{i\in A\ |\ \partial(\nu L_i)\subset \partial v\},$$

  \item the subset $\mathcal{E}(v)$ of edges of $G_L$ incident to $v$: $$\mathcal{E}(v) = \{T\in \mathcal{T}\ |\ T\subset \partial v \}.$$
\end{itemize}

\noindent The link $L(v)$ is oriented as follows. For the untwisted re-embedding $f:v\rightarrow E_{L(v)}\subset S^3$, it is enough to determine the orientations of the preferred longitudes on $f(T_i)$ for the boundary components $T_i$ of $v$.
\begin{itemize}
  \item For $i\in A(v)$, the orientation of $L(v)_i$ is induced from that of the component $L_i$ of $L$, namely, the image of the \textit{oriented} preferred longitude of $L_i$ under $f$.
  \item For $T \in \mathcal{E}(v)$ with the other endpoint $w$, let $g$ be the untwisted re-embedding of $w$. Let $\ell_v$ and $\ell_w$ be the preferred longitudes of $L (v)_T$ and $L(w)_T$, respectively. Let $\widetilde{f^{-1} (\ell_v)}$ be a parallel copy of $f^{-1}(\ell_v)$ in $\int(w)$, and vice versa. We choose the orientations of $L(v)_T$ and $L(w)_T$ so that $$lk (\widetilde{f^{-1} (\ell_v)}, g^{-1}(\ell_w)) = lk (f^{-1} (\ell_v), \widetilde{g^{-1} (\ell_w)}) = 1.$$
\end{itemize}
\noindent For a split link $L$ indexed by $\sqcup_i^k A_i$ such that each sublink $L_{A_i}$ is non-split, we set $\G_L = \sqcup_{i=1}^k \G_{L_{A_i}}$.
\end{defn}
Note that reversing both orientations of $L(v)_T$ and $L(w)_T$ also satisfies the above condition. Up to this pairwise choice for each edge, the orientations of all components of $L(v)$ are well-defined. We regard $\G_L$ modulo this ambiguity.

\begin{exmp}\label{ex:knot}
When $K$ is a knot, one can check from \cite[Definition~7]{Bud06} that the companionship $\G_K$ is a \textit{rooted tree}, namely, every edge is directed and every directed path terminates in a unique vertex, called \emph{root}. Note that the root of $\G_K$ is the component containing the boundary of $E_K$. See Figure~\ref{fig:grp-knot} for the example $K = T_{2,3}\#4_1\#\Wh^+(4_1)$.

\begin{figure}[h]
  \includesvg[scale=0.8]{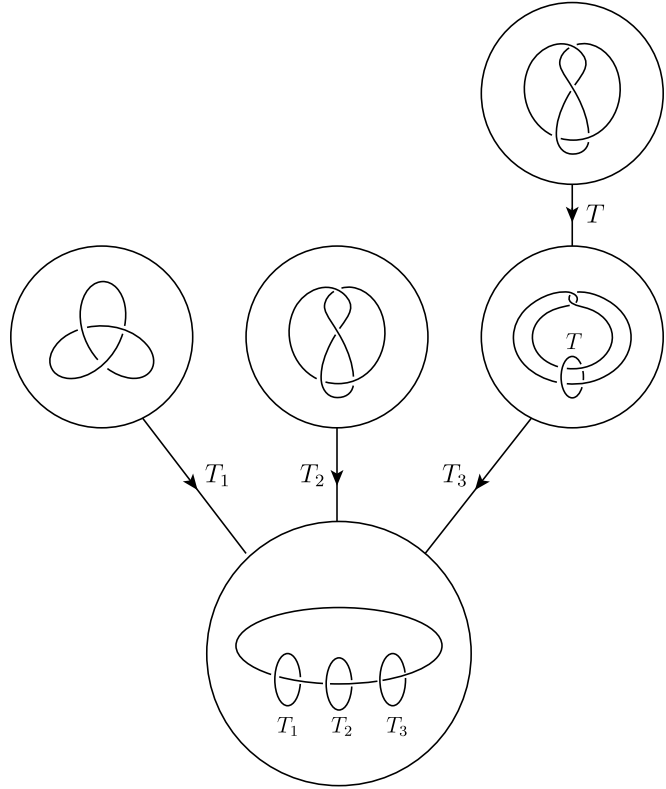} 
  \caption{The companionship graph of $K = \overline{T_{2,3}} \# 4_1 \# \Wh^+(4_1)$.}
\label{fig:grp-knot}
\end{figure}    

\end{exmp}

\begin{exmp}\label{ex:link}
    When a non-split link $L$ has more than one component, while the companionship graph may have undirected edges, $\G_L$ is still a partially directed tree. For example, the link $L$ in Figure~\ref{fig:ex-link} has $\G_L$ with an undirected edge, as described in Figure~\ref{fig:grp-link}.

    \begin{figure}[h]
  \includesvg[scale=0.8]{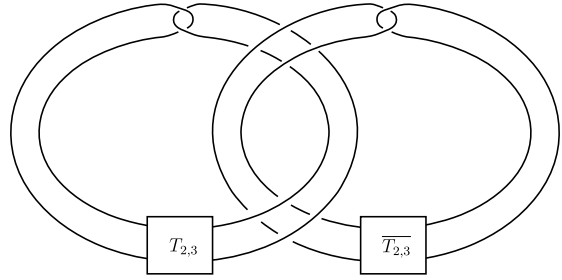}
  \caption{An example of link $L$ with companionship graph in Figure~\ref{fig:grp-link}.}
\label{fig:ex-link}
\end{figure}

\begin{figure}[h]
  \includesvg[scale=0.7]{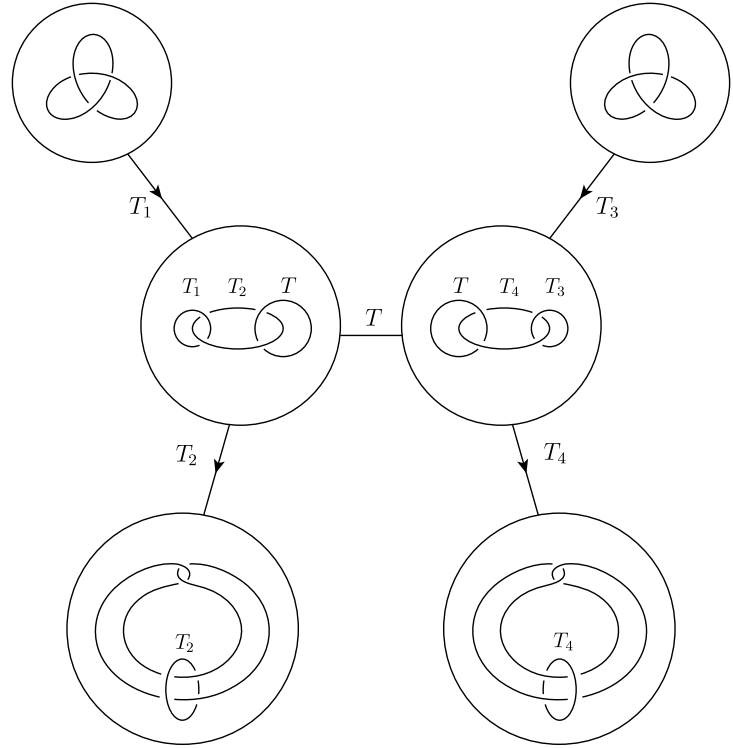}
  \caption{The companionship graph of the link $L$.}
\label{fig:grp-link}
\end{figure}
\end{exmp}

Note that the vertices of $\G_K$ in Example~\ref{ex:knot} and of $\G_L$ in Example~\ref{ex:link} are labeled by links whose exteriors are either Seifert fibered or hyperbolic. For example, the trefoil knot $T_{2, 3}$, the link in the root vertex in Figure~\ref{fig:grp-knot}, and the link in the middle in Figure~\ref{fig:grp-link} are Seifert fibered. The figure-eight knot $4_1$ and the Whitehead link at the bottom in Figure~\ref{fig:grp-link} are hyperbolic. Now we will give the exact characterization of links whose exteriors are Seifert fibered.

\begin{defn}\cite[Definition 3]{Bud06}\label{def:SF-link}
    Let $p$ and $q$ be nonzero integers and $X$ be a subset of $\{*_1, *_2\}$. A link is called a 
    \textit{Seifert fibered link} if it is isotopic to $S(p, q | X) = S_1 \cup S_2 \cup S_3$ where

\begin{itemize}
 \item $S_1 = \{(z_1, z_2)\in S^3 \ | \ z_1^p = z_2^q\}$,
  \item $S_2 = \begin{cases}
        \{(z_1,0)\in S^3 \mid z_1\in S^1\}, & \text{ if }*_1\in X, \\
        \varnothing, & \text{ otherwise},
      \end{cases}$

  \item $S_3 = \begin{cases}
        \{(0,z_2)\in S^3 \mid z_2\in S^1\}, &  \text{ if }*_2\in X, \\
        \varnothing, & \text{ otherwise}.
      \end{cases}$
\end{itemize}

\noindent See Figure~\ref{fig:Seifert-fib} for an example.

\begin{figure}[h]
  \includesvg[scale=0.7]{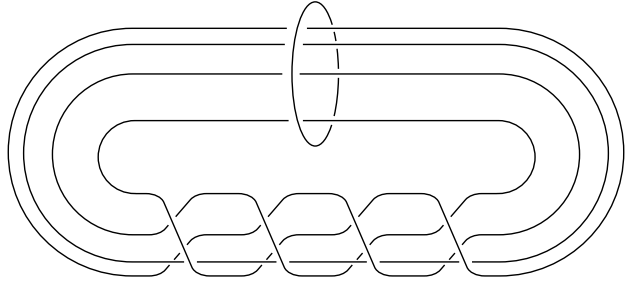}
  \caption{A Seifert fibered link $S(3, 4|{*_1, *_2})$.}
\label{fig:Seifert-fib}
\end{figure}
        
\end{defn}

Note that the complement of the link defined above is Seifert fibered. However, such links do not constitute all the links whose exteriors are Seifert fibered manifolds. Now, we introduce the other class of links with this property. 

\begin{defn}
    Let $V$ be a standard solid torus $S^1\times D^2$ in $S^3$. The core $C = S^1\times 0$ of $V$, together with $n$ parallel copies of the meridian is called the \textit{$(n+1)$-component key-chain link} $H^n$, depicted in Figure~\ref{fig:key-chain}.

\begin{figure}[h]
  \includesvg{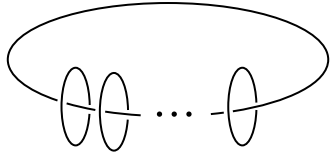}
  \caption{A key-chain link.}
\label{fig:key-chain}
\end{figure}
\end{defn}

One can see that the link in the middle of Figure~\ref{fig:grp-link} is the $3$-component key-chain link. Note that the $(n+1)$-component key-chain link $H^n$ is the same as the link $S(n,0;\{*_1\})=S(0,n;\{*_2\})$ if we allow the case where either $p$ or $q$ is $0$ in Definition~\ref{def:SF-link}. By the proposition below, the Seifert fibered links and the key-chain links are sufficient to cover all the cases when the link exteriors are Seifert fibered.

\begin{prop}\cite[Proposition~5]{Bud06}
\label{prop:seifert-fibered-link-exteriors}
Every link in $S^3$ whose exterior is Seifert fibered is isotopic to either a Seifert fibered link or a key-chain link.
\end{prop}

Now we define the splicing operation of links that can be viewed as a generalization of the satellite operation, see \cite[Corollary~2]{Bud06}.  

\begin{defn}\label{defn:splice}
    Let
    \begin{itemize}
        \item $L$ be a link in $S^3$ with an $n$-component unlinked sublink $U_1\cup \cdots \cup U_n$,
        \item $L_1,\dots,L_n$ be links in $S^3$ with a chosen component $K_i$ of each $L_i$,
        \item $\varphi_i$ be the orientation-reversing diffeomorphism between $\partial(S^3-K_i)$ and $\partial(S^3-U_i)$ which interchanges the meridian-longitude pairs.
    \end{itemize}
    The link obtained from the union of $L_i-K_i$ for $i=1,\dots,n$ and $L - (U_1\cup\dots\cup U_n)$ in $$\left (\underset{i=1}{\overset{n}{\sqcup}} (S^3-K_i) \right )\cup_{\sqcup\varphi_i}\left (S^3-(U_1\cup\dots\cup U_n)\right)\cong S^3,$$ is called the \textit{splice of $(L_1,\dots,L_n)$ and $L$ along $(K_1,\dots,K_n)$ and $(U_1,\dots,U_n)$}, and is denoted by $$(L_1,\dots,L_n)\underset{K_i\sim U_i}{\bowtie} L.$$ 
\end{defn}

We remark that even when the sublink $U_1\cup \cdots \cup U_n$ of $L$ is not unlinked, the operation still produces a link, though in a general $3$-manifold rather than in $S^3$. In our case, the resulting link is always in $S^3$. We also specify that, in the bow tie notation, the right-hand side denotes the link containing the unlinked sublink along which the splicing operation is made.

This operation is closely related to the companionship graph. If one splices two links $L_1$ and $L_2$ along the components $U_1$ and $U_2$ to obtain a link $L$, then the companionship graph $\G_L$ of $L$ can be obtained from the graph added an edge joining the vertices of $\G_{L_i}$ containing $\partial (\nu U_i)$, as described in \cite[Proposition~19]{Bud06}. Conversely, for a given link $L$, an edge-cut on the graph $\G_L$, namely, splitting the graph along an edge, is useful to obtain a splice form of $L$. Our argument mainly involves several edge-cuts.

As we can obtain a concordance between two satellite knots with the same pattern and concordant companion knots, if we have two concordant links and splice them to a link, then we obtain a concordance between the spliced links. For the detail, see Lemma~\ref{lemma:splice_concordance}.

\begin{exmp}\label{ex:splice}
Consider the link $L$ in Example~\ref{ex:link}. The link $L$ is obtained by splicing two links $L_1$ and $L_2$ in Figure~\ref{fig:link-splicing} along the unknotted components $C_1$ and $C_2$. Note that $L_1$ and $L_2$ are also obtained from a splicing operation. Compare this with Figure~\ref{fig:ex-link}.
\begin{figure}[h]
  \includesvg[scale=0.8]{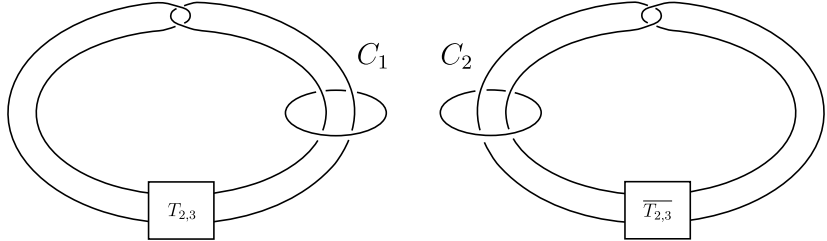}
  \caption{A splice form of the link $L$ from two links $L_1$ and $L_2$ along $C_1$ and $C_2$ is $L = L_1 \underset{C_1\sim C_2}{\bowtie}L_2$.}
\label{fig:link-splicing}
\end{figure}
If we iteratively decompose the link exteriors, then we finally obtain the companionship graph described in Figure~\ref{fig:grp-link}.
\end{exmp}

A nice way to analyze $\G_L$ is to split it along an edge $T$ and consider the connected subgraphs $\G_1$ and $\G_2$. In particular, it is always possible when $L$ is a non-split link so that $\G_L$ is a connected tree. We write this \emph{edge-cut} operation as:
$$\G_L - T = \G_1\sqcup \G_2.$$ 

We denote the corresponding $3$-manifold to such a subgraph $\G_i$ by $M(\G_i)$. Since the boundary of each piece $M(\G_i)$ is tori, by Lemma~\ref{lemma:untwisted re-embedding}, we denote such links $L(\G_1)$ and $L(\G_2)$. Let $L_i = L(\G_i)$. Regarding the index set $A$ of $L$ as the set of the boundary components of $E_L$, the link $L_i$ is canonically indexed by $A\cap M(\G_i)$ and $T$. We denote these \textit{canonical} index sets of $L_1$ and $L_2$ by $A_1$ and $A_2$, respectively. Then $A_1$ and $A_2$ satisfy:
\begin{itemize}
        \item $A_1 \cap A_2 = \{T\}$,
        \item $(A_1 \cup A_2) - \{T\} = A$.
\end{itemize}

According to the direction of $T$, at least one of the components $(L_1)_T$ and $(L_2)_T$ is unknotted, so that we can always obtain a splice form of $L$ in terms of $L_1$ and $L_2$. Moreover, the companionship graphs $\G_{L_1}$ and $\G_{L_2}$ of $L_1$ and $L_2$ can be also obtained from $\G_1$ and $\G_2$.

\begin{prop}\cite[Proposition~17]{Bud06}
\label{prop:edge-cut}
    Let $L$ be a non-split link and $T$ be an edge in $\G_L$. Write $$\G_L -T = \G_1\sqcup \G_2.$$ Then, for $i=1, 2$,
    \begin{itemize}
        \item if $T$ is undirected, then the companionship graph $\G_{L_i}$ of $L_i$ is the same as $\G_i$.
        \item if $T$ is directed from $\G_1$ to $\G_2$, then $\G_{L_1}$ is the same as $\G_1$,  and $\G_{L_2}$ is obtained from $\G_2$ by undirecting the edges of the downward consequences $DC(T)\subset \G_2$ of $T$.
    \end{itemize}
    Moreover, in any case, the splice form of $L$ is as follows:
    $$L=L_1\underset{(L_1)_T\sim (L_2)_T}{\bowtie
    } L_2.$$
\end{prop}

For a cut edge $T$ directed from $v_1\in \G_1$ to $v_2\in \G_2$, the \textit{downward consequence $DC(T)$ of $T$} is a certain subtree of $\G_2$ rooted on $v_2$, defined by Budney. Rather than giving a full definition of $DC(T)$, we only state a direct observation from its definition in the next proposition. See \cite[Definition 15]{Bud06} for the precise definition of a downward consequence.

\begin{prop}\label{prop:DC}
 If a directed edge $e$ is in $DC(T)$, then there exists a directed path from the end of $T$ to the start of $e$ in $DC(T)$.
\end{prop}

\begin{exmp}
    Consider the link $L$ in Figure~\ref{fig:ex-link} and its companionship graph $\G_L$ in Figure~\ref{fig:grp-link}. If we cut the undirected edge $T$ and write $$\G_L - T =\G_1 \sqcup \G_2,$$ where the left one is $\G_1$, then the corresponding links $L(\G_1)$ and $L(\G_2)$ are the same as $L_1$ and $L_2$ in Example~\ref{ex:splice}. See Figure~\ref{fig:link-splicing}.
\end{exmp}

\begin{remark}\label{rmk:edge-cut}
    When the cut edge $T$ is directed from $\G_1$ to $\G_2$, we note that $DC(T)$ is a subtree of $\G_2$, and hence $\G_1$ always equals $\G_{L_1}$. On the other hand, some directed edges of $\G_2$ may become undirected in $\G_{L_2}$. However, their directions are never reversed, and undirected edges in $\G_2$ never become directed in $\G_{L_2}$. These changes do not affect our arguments involving edge cuts in later proofs. For instance, the complexity $c(L_2)$, which is defined from $\G_{L_2}$, can be also obtained from $\G_2$. We will clarify such a subtlety whenever $\G_{L_2}$ is involved in our arguments.
\end{remark}

We also often cut several edges at the same time to get a splice form of the given link $L$, and analyze each subgraph. In particular, we consider the split subgraphs after removing a certain subgraph $\G'$. In this case, letting $\mathcal{E}(\G')=\{T_1,\ldots, T_n\}$ be the set of edges in $\G_L$ with exactly one endpoint in $\G'$, we consider the remainders after deleting all edges in $\mathcal{E}(\G')$. Let $\G_i$ be the other component $T_i$ meets. We write our edge-cuts as: $$\G_L - \G'= \G_1\sqcup\cdots \sqcup \G_n.$$ Let $L_i$ be the corresponding link $L(\G_i)$, and let $L'$ be $L(\G')$. Then we can separately analyze each smaller piece $\G_i$ and the link $L_i$ to study the structure of $\G_L$. In particular, we can obtain a splice form of $L$ when every edge of $\mathcal{E}(\G')$ is directed to $\G'$ as
$$L = (L_1, \ldots, L_n)\underset{(L_i)_{T_i}\sim (L')_{T_i}}{\bowtie} L'.$$ In Lemma~\ref{lemma:unlink}, we will prove that the sublink $(L')_{\mathcal{E}(\G')}$ is unlinked so that the above splice form is valid. Before that, we first provide a proposition which concerns a more certain case without assumption on the directions of the edges $\mathcal{E}(\G')$.

\begin{prop}\label{prop:many-edge-cut}
    Let $\G'$ be a connected subgraph of the companionship graph $\G_L$ of a non-split link $L$. Write $$\G_L - \G' = \G_1 \sqcup \cdots \sqcup \G_n.$$ If $\G'$ contains all boundary components of $E_L$, then each of $\G_i$ is the companionship graph of a nontrivial knot.
\end{prop}

\begin{proof}
    Since every boundary component of $E_L$ is contained in $\G'$, there are no boundary components of the $3$-manifold  $M(\G_i)$ other than the torus in $\mathcal{E}(\G')$, which joins $\G_i$ and $\G'$. Thus, the corresponding link $L(\G_i)$ is a knot. Moreover, since $M(\G_i)$, the union of some JSJ pieces, cannot be a solid torus, $L(\G_i)$ is not the unknot.
\end{proof}

We close this subsection by proving the following lemma that will be used to obtain the desired splice form later.

\begin{lem}\label{lemma:unlink}
For a non-split link $L\subset S^3$ and its companionship graph $\G_L$, let $\G'$ be a subtree of $\G_L$ and $L'$ be the link corresponding to $\G'$. Let $\mathcal{T}$ be the subset of the edges in $\G_L$ with exactly one endpoint in $\G'$. If each $T_i\in\mathcal{T}$ is directed from outside $\G'$ to $\G'$, then the sublink $(L')_\mathcal{T}$ of $L'$ is an unlink.
\end{lem}

For a link $L$ indexed by $A$, let $\overline{\mathcal{U}}_L$ be the collection of every subset $B\subset A$ such that the sublink $L_B$ of $L$ indexed by $B$ is an unlink. By definition, to prove Lemma~\ref{lemma:unlink}, it is enough to show that $\mathcal{T}\in \overline{\mathcal{U}}_{L'}$. The proof relies on the iteration of the following proposition, called the \textit{global Brunnian property}, by Budney:

\begin{prop}\cite[Proposition~17]{Bud06}
\label{prop:Brunnian}
    Let $L$ be a non-split link indexed by $A$ and $T$ be an edge in $\G_L$ and write $$\G_L -T = \G_1\sqcup \G_2.$$ Let $A_1$ and $A_2$ be the canonical index sets from $A$ and $T$. Then for any $B\in \overline{\mathcal{U}}_L$, one of the following holds:
    \begin{enumerate}
        \item $B\cap A_1 \cup \{T\}\in \overline{\mathcal{U}}_{L_1}$ and $B\cap A_2 \in \overline{\mathcal{U}}_{L_2}$,
        \item $B\cap A_2 \cup \{T\}\in \overline{\mathcal{U}}_{L_2}$ and $B\cap A_1 \in \overline{\mathcal{U}}_{L_1}$.
    \end{enumerate}
\end{prop}
 
\begin{proof}[Proof of Lemma~\ref{lemma:unlink}]

Let $T_1,\ldots, T_n$ be the elements of $\mathcal{T}$. Let $\G^0_1 = \G_L$ and $\G^0_2 = \varnothing$. We define two subgraphs $\G_1^i$ and $\G_2^i$ inductively by $$\G_1^i-T_{i+1}=\G_1^{i+1} \sqcup \G_2^{i+1},$$ where $\G_1^{i+1}$ is the component containing $\G'$, so that $\G_1^n=\G'$. See Figure~\ref{fig:unlinked-lemma} for a schematic picture of the situation. For simplicity let $L_1^i$ and $L_2^i$ denote the corresponding links $L(\G_1^i)$ and $L(\G_2^i)$ to $\G_1^i$ and $\G_2^i$, respectively. Let $A_1^i$ and  $A_2^i$ be the canonical index sets of $L_1^i$ and $L_2^i$, respectively. Observe from Proposition~\ref{prop:edge-cut} that the companionship graphs $\G_{L_1^i}$ and $\G_{L_2^i}$ are the same as $\G_1^i$ and $\G_2^i$ by ignoring all directions on the edges.
Then for each $i=1,\ldots, n$, $A_1^i$ and $A_2^i$ satisfy
\begin{itemize}
            \item $A_1^i \cap A_2^i = \{T_i\}$,
        \item $(A_1^i \cup A_2^i) - \{T_i\} = A_1^{i-1}$.
\end{itemize}

\begin{figure}[h]
  \includesvg{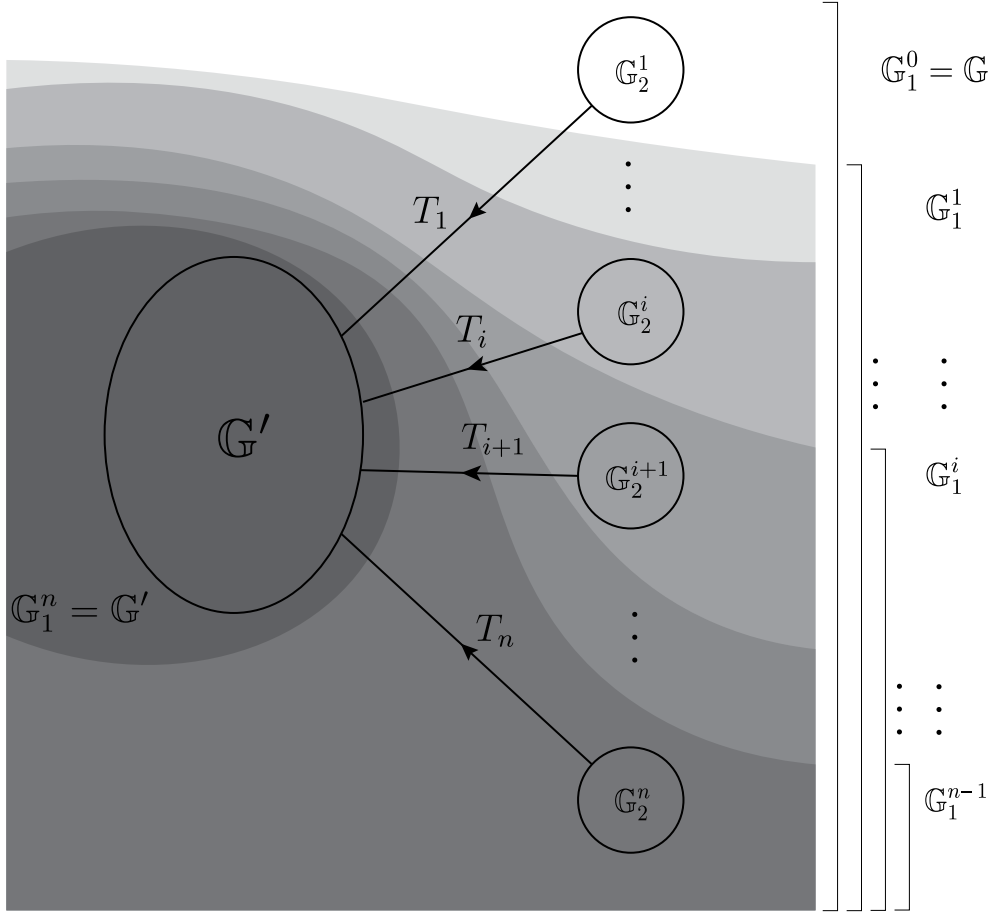}
  \caption{The schematic picture of the situation in the proof of Lemma~\ref{lemma:unlink}. The region containing $\G_1^i$ is shaded darker as $i$ increases.}
\label{fig:unlinked-lemma}
\end{figure}

We claim that $T_{i+1}, \ldots, T_n$ in $\G_1^i$ remain directed to $\G'$ as edges in $\G_{L_1^i}$. By Proposition~\ref{prop:edge-cut}, it is enough to check that $T_{i+1},\ldots, T_n$ in $\G_1^i$ are not in the downward consequence $DC(T_i)\subset \G_1^i$. Since $T_{i+1},\ldots, T_n$ are directed to $\G'$ as edges in the subgraph $\G_1^i$, if some $T_k$ among $T_{i+1},\dots,T_n$ were in $DC(T_i)$, then there would be a directed path from the end of $T_i$ to the start of $T_k$ by Proposition~\ref{prop:DC}. Then, however, $T_k$ could not split $\G_1^{k-1}$, which is a contradiction. Thus, all $T_{i+1},\ldots,T_n$ are not in $DC(T_i)$ and hence, they are still directed to $\G'$ as edges of $\G_{L_1^i}$.

To prove that $L_{\mathcal{T}}'$ is unlinked, as mentioned before, we claim that $\{T_1,\ldots,T_n\}\in\overline{\mathcal{U}}_{L_1^n} = \overline{\mathcal{U}}_{L'}$. We proceed by induction on $n$. For $n=1$, take $B\in \overline{\mathcal{U}}_L$ as the empty set $\varnothing$. Then by Proposition~\ref{prop:Brunnian}, one of the following holds:
\begin{enumerate}
    \item $\varnothing \cap A_1^1 \cup \{T_1\} = \{T_1\}\in \overline{\mathcal{U}}_{L_1^1}$ and $\varnothing\cap A_2^1=\varnothing\in \overline{\mathcal{U}}_{L_2^1}$,
    \item $\varnothing \cap A_2^1 \cup \{T_1\} = \{T_1\}\in \overline{\mathcal{U}}_{L_2^1}$ and $\varnothing\cap A_1^1=\varnothing\in \overline{\mathcal{U}}_{L_1^1}$.
\end{enumerate}
Recall that we have shown that $T_i$ exits $\G_2^i$ as the edge in the companionship graph $\G_{L_1^{i-1}}$. Since $T_1$ is exiting $\G_{L_2^1}$, by the definition of a directed edge, $T_1$ does not bound a solid torus in $S^3$ on $\G_2^1$ side. Thus, the component $(L_2^1)_{T_1}$ of $L_2^1$ cannot be unknotted, i.e., $\{T_1\}\notin\overline{\mathcal{U}}_{L_2^1}$. Thus, the condition (1) holds.

Now suppose $\{T_1,\dots,T_k\}\in\overline{\mathcal{U}}_{L_1^k}$ for $k < n$. Take $B=\{T_1,\ldots, T_k\}$. Then, by Proposition~\ref{prop:Brunnian}, one of the following holds:
\begin{enumerate}
    \item $\{T_1,\dots,T_{k+1}\}\in\overline{\mathcal{U}}_{L_1^{k+1}}$ and $\varnothing\in\overline{\mathcal{U}}_{L_2^{k+1}}$,
    \item $\{T_{k+1}\}\in\overline{\mathcal{U}}_{L_2^{k+1}}$ and $\{T_1,\dots,T_k\}\in\overline{\mathcal{U}}_{L_1^{k+1}}$.
\end{enumerate}
Similarly, since $T_{k+1}$ is exiting $\G_{L_2^{k+1}}$, we see that $\{T_{k+1}\}\notin\overline{\mathcal{U}}_{L_2^{k+1}}$ hence the condition (1) holds. Therefore, $\{T_1,\dots,T_n\}\in\overline{\mathcal{U}}_{L'}$, i.e, the sublink $(L')_\mathcal{T}$ of $L'$ is an unlink, as desired.
\end{proof}

\subsection{Complexity of links}

In this subsection, we will introduce \textit{complexity} of links. This quantity is the key ingredient for our inductive argument in the proof of Theorem~\ref{thm:link}. In particular, we will prove its monotonicity with respect to taking a sublink, and well-ordering of the set of complexities of all links. We first recall the notion of the Gromov norm for link exteriors.

\begin{defn}\cite[Definition~18]{Bud06}\label{def:Gromov norm}
    Let $L\subset S^3$ be a non-split link with its companionship graph $\G_L$. For the link exterior $E_L$ of $L$, the \textit{Gromov norm} $||E_L||$ is defined as
    $$||E_L|| = \sum\limits_{v\in \G_L} \operatorname{Vol}(v),$$
    where $\operatorname{Vol}(\cdot)$ is the hyperbolic volume for hyperbolic pieces and zero for Seifert fibered pieces.
\end{defn}

\noindent Now we define the complexity of link $L$ by combining the Gromov norm and the number of JSJ pieces of $E_L$.

\begin{defn}
\label{defn:complexity}
Let $L \subset S^3$ be a non-split link. We define the \emph{complexity $c(L)$ of $L$} as the pair $$c(L) = \left ( ||E_L||, \ \# \G_L \right) \in \mathbb{R}_{\geq 0} \times \mathbb{N},$$ where $\# \G_L$ denotes the number of vertices of $\G_L$. For a split link $L$ which consists of non-split sublinks $L_1, \ldots, L_n$, we set $$c(L) = \sum_{i=1}^n c(L_i).$$
\end{defn}

Now we consider the set of complexities of links as ordered as a subset of $\R_{\geq 0}\times \mathbb{N}$, endowed with the lexicographic order. In other words, for two links $L_1$ and $L_2$, their complexities are ordered as follows:

$$c(L_1) < c(L_2) \iff ( ||E_{L_1}|| < ||E_{L_2}|| ) \ \text{or} \ ( ||E_{L_1}|| = ||E_{L_2}|| \ \text{and} \ \#\G_{L_1} < \#\G_{L_2} ). $$

\noindent First, we prove the monotonicity of the complexity of links with respect to taking a sublink.

\begin{lem}\label{lemma:decreasing_complexity}
Let $L\subset S^3$ be a link and $L'\subset L$ a sublink. Then $c(L')\leq c(L)$.
\end{lem}
\begin{proof}
By definition, we may assume that $L$ is non-split and $L'$ is obtained from $L$ by removing a single component $K$. Let $v\in\G_L$ be the vertex containing the boundary component $\partial(\nu K)$. Let $M'$ be the $3$-manifold obtained from the JSJ piece $v=E_{L(v)}$ by filling the component of $L(v)$ corresponding to $\partial(\nu K)$. Then $$E_{L'} = E_L \cup \nu K = (E_L-v)\cup (v\cup \nu K) = (E_L-v)\cup M'.$$

Suppose that $v$ is hyperbolic. Then, by \cite[Theorem~6.5.6]{Thu97} (see also \cite[p. 356]{Bud06}), we have that $||M'||<||v||$. Therefore, we see that
$$
||E_{L'}||\leq ||E_L-v||+||M'||<||E_L-v||+||v||=||E_L||,
$$
which implies that $c(L')<c(L)$. Note that the first inequality comes from the subadditivity of the Gromov norm with respect to gluing along a torus boundary \cite[Theorem~7.6]{Fri17}. See also \cite{Gro82, BBFIPP14, Kue15}.

Suppose now that $v$ is Seifert fibered. Then, by Proposition~\ref{prop:seifert-fibered-link-exteriors}, $L(v)$ is either a Seifert fibered link or a key-chain link. Since the link $L(M')$, given by the untwisted re-embedding of $M'$, is obtained from $L(v)$ by deleting one component, $M'$ is still Seifert fibered. Deleting a component of a link is equivalent to splicing the unknot along the component. Now, \cite[Proposition~19]{Bud06} describes a procedure to obtain $\G_{L'}$ from the labeled graph given by substituting the vertex $v$ in $\G_L$ with a vertex $v'$ labeled by $L(M')$. It is immediate to check that the complexity is non-increasing at each step of this procedure since such a step only involves Seifert fibered vertices and splittings of the graph with respect to the prime decomposition, and does not increase the number of vertices.
\end{proof}

The goal of the remaining part of this subsection is to show that the set of complexities of links is well-ordered, so that we can proceed with the induction on the complexity in the proof of Theorem~\ref{thm:link}.

\begin{lem}
\label{lemma:well-ordering}
Let $Y \subset \mathbb{R}_{>0}$ be a well-ordered subset. Then the set $$ X = \left \{  \sum_{i=1}^{m} a_i \cdot y_i \ \vert \ m \geq 0, \ a_i \in \mathbb{N}, \ y_i \in Y \right \} \subset \mathbb{R}_{\ge 0}$$ is well-ordered.
\end{lem}

\begin{proof}
Let $\{x_m\}_{m\in \mathbb{N}}$ be a decreasing sequence in $X$. We want to prove that it stabilizes, i.e., there exists some $m_0$ such that for all $m \geq m_0$ we have $x_m = x_{m_0}$. Observe that since the sequence is decreasing, it is sufficient to find a stabilizing subsequence.

It is clear that if $x_m =0$ for some $m$, then it stabilizes, so we assume that $x_m >0$. Let $z = \min (Y)>0$ and let $k = \lfloor x_0 /z \rfloor$. Observe that every $x_m$ is written as a sum of at most $k$ elements in $Y$ since $x_m \le x_0 < (k+1)z$. In other words, for all $m$, $x_m = y_{1,m} + \cdots + y_{k,m}$, where $y_{i,m} \in Y \cup \{0\}$.

We prove that $\{x_m\}_{m\in \mathbb{N}}$ stabilizes by induction on $k \geq 1$. For the base case $k=1$, $x_m \in Y$ for all $m$ and it stabilizes since $Y$ is well-ordered. For the inductive case, assume that $k\geq 2$. Then write $$x_m = y_{1,m} + \cdots + y_{k,m} \quad \text{with} \ y_{i,m} \in Y \cup \{0\} \ \text{and} \ y_{i,m} \geq y_{i+1,m} \ \text{ for all } i . $$ 

Consider the sequence $\{y_{k,m}\}_{m \in \mathbb{N}}$. If this sequence admits a subsequence $\{y_{k,m_j}\}$ with the property $y_{k,m_j} \geq y_{k,m_{j+1}}$, then it implies that at some point it stabilizes at a value of $y_k \in Y\cup \{0\}$. Then the sequence $\{x_{m_j} -y_k \}$ is \emph{definitively} composed by elements of $X$ which can be written as a sum of at most $k-1$ elements of $Y\cup \{0\}$. By induction, $\{x_{m_j}-y_k\}$ stabilizes and hence so does $\{x_m\}$.

Suppose now that $\{y_{k,m}\}$ admits no infinite decreasing subsequence. Then necessarily we can extract a strictly increasing subsequence $y_{k,m_j} < y_{k,m_{j+1}}$. Now we do the same with the sequence $\{y_{k-1, m}\}$: 
\begin{itemize}
    \item if it admits an infinite decreasing subsequence, it stabilizes to $y_{k-1} \in Y\cup \{0\}$ and we consider $\{x_{m_j} - y_{k-1} \}$,
    \item otherwise, we get a strictly increasing subsequence $y_{k-1, m_j} < y_{k-1, m_{j+1}}$.
\end{itemize}

This process ends when we consider $\{y_{1, m}\}$. If it does not stabilize, then we extract a subsequence $y_{1, m_j} < y_{1, m_{j+1}}$. But then $y_{i, m} < y_{i, m+1}$ for all $i$ and $m$. Hence $x_m < x_{m+1}$, leading a contradiction. Therefore, $\{x_m\}$ stabilizes, and hence $X$ is well-ordered.
\end{proof}

Since the set of volumes of hyperbolic $3$-manifolds is known to be well-ordered as a subset of $\mathbb{R}_{>0}$ by Thurston--J{\o}rgensen theorem \cite[Corollary~6.6.1]{Thu97}, we have the following corollary from this fact and Lemma~\ref{lemma:well-ordering}.

\begin{cor}
\label{cor:-TJ-volume}
The set $$\operatorname{Vol} = \left \{ \sum_{i=1}^{n} a_i \operatorname{Vol}(M_i) : a_i \in \mathbb{N}, \ M_i \ \text{is a hyperbolic} \ 3\text{-manifold} \right \} \subset \mathbb{R}_{\geq 0}$$
is well-ordered.    
\end{cor}

If $(X,<_X)$ and $(Y,<_Y)$ are two well-ordered sets, then so is $(X\times Y,<_{lex})$. Therefore, we have the following immediate conclusion from Corollary~\ref{cor:-TJ-volume}.

\begin{cor}\label{cor:well-ordered}
The set $(\operatorname{Vol} \times \mathbb{N}, <_{lex})$ is well-ordered. In particular, $$\{ c (L) \ | \ L \subset S^3 \ \text{a link} \}$$ is a well-ordered subset of $\operatorname{Vol} \times \mathbb{N}$.
\end{cor}

\noindent Therefore, we can use complexity of links for an induction argument. We always use $<$ instead of $<_{lex}$.

\subsection{Smith theory}\label{subsection-Smith}

By the resolution of the Smith conjecture \cite{Wal69, MB84}, the fixed point set of a given finite order self-diffeomorphism of the $3$-sphere is well-understood. In this subsection, we recall some fundamental results and prove a lemma.

\begin{thm}\cite[Section~3, Theorems~5.1 and 5.2]{Bre72}
\label{thm:smith-hom-sphere}
Let $M$ be a compact smooth $\Z_pHS^n$ (resp. $\Z_p HB^n$) for a prime $p$. If $\tau\in\Diff(M)$ is of finite order $p$, then $\Fix(\tau)$ is either \begin{itemize}
    \item empty,
    \item a $\Z_pHS^k$ (resp. $\Z_pHB^k$) for some $k\leq n$.
\end{itemize}
\end{thm}

\begin{cor}
\label{thm:smith-disk}
For a prime $p$, let $\tau\in\Diff^+(D^2)$ be a map of order $p$. Then $\Fix(\tau)$ is a single point.
\end{cor}

\begin{thm}\cite{MB84}\cite[Section~10]{Kaw96}
\label{thm:smith-orthogonal}
Every element $\tau\in\Diff(S^3)$ of finite order is conjugate to an element of $\Or(4)$. In particular,
\begin{enumerate}
    \item If $\tau$ is orientation-preserving, then $\Fix(\tau)$ is either $\varnothing$, or an unknotted $S^1$, or $S^3$.
    \item If $\tau$ is orientation-reversing, then $\Fix(\tau)$ is either $S^0$ or $S^2$.
\end{enumerate}
Moreover, when $\Fix(\tau)$ is an unknotted $S^1$, $\tau$ acts as a rotation with axis $\Fix(\tau)$. 
\end{thm}

\begin{thm}
\label{thm:smith-orthogonal}
Every element $\tau\in\Diff(S^3)$ of finite order is conjugate to an element of $\Or(4)$. In particular,
\begin{enumerate}
    \item If $\tau$ is orientation-preserving, then $\Fix(\tau)$ is either $\varnothing$, or an unknotted $S^1$, or $S^3$.
    \item If $\tau$ is orientation-reversing, then $\Fix(\tau)$ is either $S^0$ or $S^2$.
\end{enumerate}
Moreover, when $\Fix(\tau)$ is an unknotted $S^1$, $\tau$ acts as a rotation with axis $\Fix(\tau)$. 
\end{thm}

The above results will be used several times to understand fixed point sets of certain finite order maps. Now we prove a lemma. This lemma will be used in Lemma~\ref{lemma:elementary-graph} to obtain the cardinality of the orbit of each component of a hyperbolic link, with respect to a rotation of an unknotted component of the link. See Figure~\ref{fig:LUA} for examples.

\begin{lem}\label{lemma:periodic_unlink}
Let $L\subset S^3$ be an unlink and $\rho:S^3\to S^3$ be a rotation of order $n>1$ around an unknotted axis $A$, such that $A\cap L=\varnothing$ and $\rho(L)=L$.
For each component $K$ of $L$, set $$n_K=\min\{i\geq1\;|\;\rho^i(K)=K\}.$$ Then either of the following holds:\begin{itemize}
    \item $L\cup A$ is not a hyperbolic link,
    \item $n_K=n$ for any component $K$ of $L$.
\end{itemize} 
\end{lem}

\begin{proof}
Note that $n_K\le n$ for any component $K$ of $L$. We will prove that if $n_K<n$ for some component $K$, then the exterior of $L\cup A$ contains an essential annulus (see \cite[Section~9.4.7]{Mar23}). Then, by \cite[Proposition~9.4.18]{Mar23}, the exterior of $L\cup A$ is not hyperbolic.

We first consider the case when $L$ has only one component. Then, by \cite[Theorem~4]{Edm84}, $L$ bounds a $\rho$-invariant disk $D$ in $S^3$. By Corollary~\ref{thm:smith-disk}, $A\cap D$ is a single point. Observe that $D-(L\cup A)$ is an annulus, which is essential by \cite[Corollary~9.3.10]{Mar23}. Hence, $L\cup A$ is not a hyperbolic link.

Now we consider the case when $|L|\geq 2$. Suppose there exists a component $K$ of $L$ such that $n_K<n$. Let $\tau=\rho^{n_K}$. Since $L$ is an unlink, there exists a $2$-sphere $S\subset S^3$ such that $S^3-S$ is the union of two $3$-balls, one containing $K$ and the other containing $L-K$. In particular, $S$ does not bound a ball in the exterior $E_L$ of the link $L$. According to \cite[Theorem~3.1]{Dun85}, we can isotope $S$ in $E_L$ so that either $\tau(S)=S$ or $\tau(S)\cap S=\varnothing$. 

We claim that the former is the only possible case. Suppose by contradiction that the latter holds, and let $E_L-S=Y\sqcup Z$, where the boundary component corresponding to $K$ is contained in $Y$. Since $\tau(K)=K$, we have that either $\tau(Y)\subsetneq\int Y$ or $Y\subsetneq\int \tau(Y)$. In the first case, observe that it implies that $\tau^n(S)\subsetneq\int Y$, which is a contradiction since $\tau^n$ is the identity. The second case similarly leads to a contradiction. Therefore, we have that $\tau(S)=S$.

Since $S$ is $\tau$-invariant, it can be $\tau$-equivariantly decomposed into two disks. By Corollary~\ref{thm:smith-disk}, we have that $A\cap S\cong S^0$, and hence that $S\cap (S^3-(L\cup A))=S-A$ is an annulus, which is essential again by \cite[Corollary~9.3.10]{Mar23}. Therefore, $L\cup A$ is not a hyperbolic link.
\end{proof}

    \begin{figure}
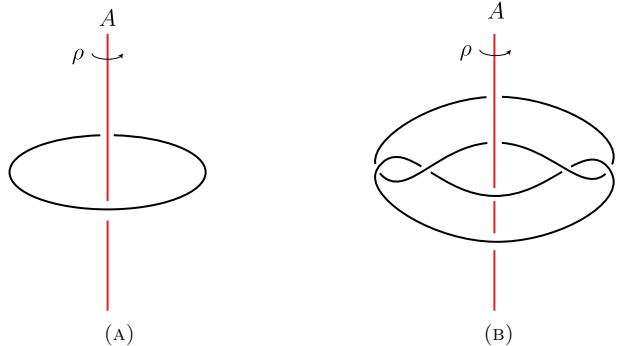
%
    \centering
    \subfloat[]
    {{\includesvg{./figures/LUA-Hopf.svg} }}%
    \qquad\qquad\qquad
    \subfloat[]
    {{\includesvg{./figures/LUA-Borromean.svg} }}%
    \caption{In both of cases (A) and (B), let $\rho$ be the $180^\circ$ rotation along the unknotted axis $A$ so that the order $n$ of $\rho$ is $2$. In (A), $L\cup A$ forms the Hopf link, which is not hyperbolic. In this case, $n_K = 1$. In (B), $L\cup A$ forms the Borromean ring, which is hyperbolic and $n=n_K$ for any component $K$ of $L$.}%
    \label{fig:LUA}%
\end{figure}


\section{Action of amphichiral maps on JSJ decompositions}
\label{sec:3}

Let $(L, f)$ be an \emph{amphichiral link}, i.e., a link $L\subset S^3$ together with an orientation-reversing self-diffeomorphism $f$ of $S^3$ which sends each component of $L$ to itself. Let $\G_L$ be the companionship graph of $L$. In this section, we observe how $f$ acts on $\G_L$ and prove several properties of this action. In particular, we prove that $(L,f)$ admits an amphichiral involution under certain conditions on $\G_L$ and the action of $f$. We define \textit{coherently} and \textit{incoherently} directed edges for a fixed edge in $\G_L$ which play key roles in the proof of Theorem~\ref{thm:link}. We also introduce notions of \textit{totally coherent} and \textit{properly incoherent} JSJ structure of a negative amphichiral knot exterior, which appear in Theorems~\ref{thm:SNACK} and \ref{thm:SNACK-concordance}.
  
Let $\mathcal{T}$ be the set of JSJ tori of $E_L$. Since $f(\mathcal{T)}=\{f(T)\}_{T\in\mathcal{T}}$ gives another minimal JSJ decomposition of $f(E_L)=E_L$ and such a minimal decomposition is unique up to isotopy, we can replace $f$ with an isotopic map so that $f(\mathcal{T})=\mathcal{T}$. Thus, $f$ acts on the companionship graph $\G_L$ in the following fashion:
  
\begin{prop}\label{prop:graph_properties}
Let $(L, f)$ be a non-split amphichiral link in $S^3$. Then the action of $f$ on the companionship graph $\G_L$ of $L$ satisfies the following:
\begin{enumerate}
\item $f$ acts by an automorphism on the underlying JSJ graph $G_L$.
    \item If a vertex $v$ is labeled by the link $L(v)$ then $f(v)$ is labeled by the mirror $\overline{L(v)}$.
    \item If $v$ is a vertex containing a component of $\partial E_L$ then $f(v)=v$.
    \item If $T$ is an edge in $\G_L$ connecting $v$ and $w$ and $f(T)=T$, then $f(v)=v$ and $f(w)=w$.
    \item If there are no fixed edges, then there exists precisely one fixed vertex, which contains all boundary components of $E_L$.
\end{enumerate}
\end{prop}

\begin{proof}
    (1)-(3) are clear from the definitions of $f$ and $\G_L$. For (4), we can write $$\G_L - T = \G_1\sqcup \G_2,$$ where $v\in \G_1$ and $w\in \G_2$. If $f(v)=w$ and $f(w)=v$, then $f(\G_1)=\G_2$ and $f(\G_2)=\G_1$. Then, there are no fixed vertices. However, since each boundary component of $E_L$ is fixed by $f$, by (3), there must exist at least one fixed vertex, which is a contradiction. 
    
    For (5), let $v$ be a vertex containing a boundary component of $E_L$. Then, from (3), $v$ is fixed by $f$. Suppose there were another fixed vertex $w$. Since $\G_L$ is a tree, there exists a unique path connecting $v$ and $w$. Moreover, such a path must be fixed by $f$ since $v$ and $w$ are fixed. This contradicts that $f$ fixes no edges of $\G_L$. Thus, $v$ is the unique fixed vertex. Therefore, again from (3), there cannot be other vertices containing a boundary component of $E_L$, i.e., $v$ contains all the boundary components.
\end{proof}

Note that since $f$ fixes the boundary component of $E_L$, if $L$ is split so that $\G_L$ has several components, then $f$ acts on each component separately. Thus, from now on, we only consider non-split amphichiral links and always assume that $f$ acts on the JSJ decomposition of $E_L$ for an amphichiral link and its companionship graph $\G_L$. 

Now we specify whether the amphichiral map $f$ of an oriented amphichiral link preserves or reverses the orientation of each component as follows.

\begin{defn}\label{def:(p,q)}
An oriented link $L = K_1 \cup \cdots \cup K_{p+q}$ is said to be \textit{$(p,q)$-amphichiral} if there exists an orientation-reversing map $f:S^3\to S^3$ such that 
$$f(K_i) = \begin{cases}
-K_i, & \text{for} \ 1\leq i\leq p, \\
+K_i, & \text{for} \ p+1\leq i\leq p+q.
\end{cases}$$
\noindent If such $f$ is an involution, we say that $L$ is \textit{strongly $(p,q)$-amphichiral}.
\end{defn}

\noindent We also use the same terminology for $(L, f)$ when we want to specify the amphichiral map $f$. By definition, an $n$-component negative amphichiral link $L$ is $(n, 0)$-amphichiral. We occasionally write $(-)$- (resp. $(+)$)-amphichiral to write negative (resp. positive) amphichiral for simplicity. For a $(p,q)$-amphichiral link $(L,f)$ indexed by $A$, we write $A = A^- \sqcup A^+$, where the sublinks $(L_{A^{\mp}}, f)$ are $(\mp)$-amphichiral, respectively.

\begin{defn}\label{def:reduced}
    Let $(L, f)$ be a $(p,q)$-amphichiral link in $S^3$. We say that $(L, f)$ is \textit{reduced} if the order of the action of $f$ on $\G_L$ and the order $[f|_v]$ in the mapping class group of the JSJ piece $v$ for each hyperbolic vertex $v$ of $\G_L$ are both a power of $2$.
\end{defn}

\begin{remark}\label{rmk:reduced}
Let $(L,f)$ be a $(p,q)$-amphichiral link. Observe that for every odd $n$, $f^n$ is an orientation-reversing map that $(L,f^n)$ is still $(p,q)$-amphichiral. Therefore, by replacing $f$ with $f^n$ for some odd $n$, we can always assume that $(L,f)$ is reduced. From now on, we will only consider this case.
\end{remark}

Let $\G \subset \G_L$ be a subgraph. By Lemma~\ref{lemma:untwisted re-embedding}, the corresponding $3$-manifold $M(\G)$ to $\G$ is diffeomorphic to the exterior of a link $L(\G) \subset S^3$. Suppose that $f(\G) = \G$. In particular, $f$ fixes $M(\G)$ in $S^3$. Then, by \cite[Corollary 7.3]{Liv23}, the restriction $f|_{M(\G)}$ extends via the untwisted re-embedding to a map $$F_\G : (S^3, L(\G)) \to (S^3, L(\G)),$$ since $f|_{M(\G)}$ preserves the peripheral structure of $L(\G)$. It is not difficult to check that if $f|_{M(\G)}$ is isotopic to a finite order map, then, up to isotopy, one can pick the extension $F_\G$ to be a map of the same order. For example, one may apply the same argument in \cite[Section~2]{Har80} for the case where $f|_{M(\G)}$ is an involution, to a general finite order case. (See also \cite[Lemma~1.3]{Har80b}.) Note that $L(\G)$ is not amphichiral in general since $F_\G$ does not preserve the components corresponding to not fixed edges in $\G$ by $f$.

Recall from Proposition~\ref{prop:graph_properties} (3) that if $f$ fixes only a single vertex $v$ of $\G_L$, then $v$ is the JSJ piece containing all boundary components. Consequently, the corresponding link $L(v)$ is indexed by $A\cup \mathcal{E}(v)$ where $A$ is the index set of $L$. Now we consider this case for certain $(p,q)$-amphichiral links. We prove the below lemma, which will be used later to obtain a $(p,q)$-amphichiral involution of a certain $(p,q)$-amphichiral link $L$. In particular, for the case where $p=1$, the involution $g$ with $\Fix(g)\cong S^0$ will be obtained for smaller pieces after edge-cuts of some $\G_L$, and be glued together along tori to get an amphichiral involution on the whole $E_L$, and hence on $(S^3, L)$ in Lemma~\ref{lemma:coherent_root}.

\begin{lem}\label{lemma:elementary-graph}
  Let $(L,f)$ be a reduced $(p, q)$-amphichiral link with $p\ge 1$, indexed by $A = A^- \sqcup A^+$. If $f$ fixes only a single vertex $v$ of $\G_L$ and the sublink $L(v)_{A^+\cup \mathcal{E}(v)}$ of $L(v)$ is unlinked, then there exists an involution $g$ on $S^3$ such that $(L,g)$ is strongly $(p,q)$-amphichiral. Moreover, if $p=1$, one can choose $g$ so that $\Fix(g)\cong S^0$.
\end{lem}

\begin{proof}

Recall that $L(v)$ is indexed by $A\cup \mathcal{E}(v)$. We write $L(v)$ as: 
$$L(v) = K_1^- \cup \cdots \cup K_p^- \cup K_1^+ \cup \cdots \cup K_q^+ \cup K_1^e \cup \cdots \cup K_k^e,$$ where $K_i^-$, $K_i^+$, and $K_i^e$ are the components of the sublinks $L(v)_{A^-}$, $L(v)_{A^+}$, and $L(v)_{\mathcal{E}(v)}$, respectively. Let $F_v:(S^3,L(v))\to(S^3,L(v))$ be the map obtained from the restriction $f$ to $v$. Then the sublinks $(L(v)_{A^{\mp}}, F_v)$ are $(\mp)$-amphichiral, respectively.

Since $v$ is the unique fixed vertex of $\G_L$, $f$ fixes no edges in $\G_L$ by Proposition~\ref{prop:graph_properties} (3). In particular, this fact with the assumption for $(L,f)$ being reduced implies that $k=|\mathcal{E}(v)|$ is even. Thus, we can reindex $L(v)_{\mathcal{E}(v)}$ so that $F_v$ maps $K_i^e$ to $K_{i+k/2}^e$ for $1\leq i\leq k/2$. Note that $F_v(K_{i+k/2}^e)$ may not be $K_i^e$. Now we have two cases to consider: the JSJ piece $v$ is either Seifert fibered or hyperbolic.

\vspace{0.25em}
\noindent \textbf{Case 1.} $v$ is Seifert fibered.
\vspace{0.25em}

According to Proposition~\ref{prop:seifert-fibered-link-exteriors}, the link $L(v)$ is either a Seifert fibered link or a key-chain link.

\vspace{0.25em}
\noindent $\blacktriangleright$ \textbf{Case 1.a.} $L(v)$ is a Seifert fibered link $S(r, s|X)$.
\vspace{0.25em}

We prove that this case never occurs. Write $S(r,s|X)=S_1\cup S_2\cup S_3$ in form in Definition~\ref{def:SF-link}. If $\gcd(r,s)\neq 1$, then $S_1$ is a cable of another Seifert fibered link so that $v$ contains an incompressible torus not parallel to a component of $\partial v$. Thus, $\gcd(r,s)=1$. Note that
 \begin{itemize}
     \item $\lk(S_1, S_2) = r$,
     \item $\lk(S_1, S_3)=s$,
     \item $\lk(S_2, S_3)=1$.
     \end{itemize}
     
     Suppose $S_1$ is a component of $L(v)_{\mathcal{E}(v)}$. Since $|\mathcal{E}(v)|$ is even, $S_2\cup S_3\neq \varnothing$ and exactly one of $S_2$ or $S_3$ is another component of $\mathcal{E}(v)$. However, then $L(v)_{\mathcal{E}(v)}$ is not unlinked.
     
     Thus, $S_1$ is not indexed by $\mathcal{E}(v)$, but by $A$. Then, since $L(v)_A$ is amphichiral, the component $S_1$ should be the unknot, i.e., $r$ or $s$ is $1$. Assume $s=1$. Since $L(v)_{A^+\cup \mathcal{E}(v)}$ is unlinked and $|\mathcal{E}(v)|$ is even, the only possible case is $\mathcal{E}(v)=\varnothing$ and exactly one of $S_2$ or $S_3$ is a component of $L(v)_A$, while the other is empty.
     
     By symmetry, we may assume $S_2$ is nonempty. By a linking number argument again, $S_1$ and $S_2$ cannot be both in $L(v)_{A^-}$. By symmetry, we may assume that $S_1\subset L(v)_{A^-}$ and $S_2\subset L(v)_{A^+}$. Then, for any negative amphichiral knot $K$, we can obtain a negative amphichiral map of $K_{r,1}$. However, it is known that any cable knot is never amphichiral by \cite[Lemma~4.2]{PS20}, so it is impossible.

\vspace{0.25em}
\noindent $\blacktriangleright$ \textbf{Case 1.b.} $L(v)$ is a key-chain link $H^m$.
\vspace{0.25em}

Write $H^{m}=C\cup U_1 \cup\dots\cup U_m$, where $C$ is the core of the key-chain link and the $U_i$ are the rings. Note that $p+q+k = m + 1$. If $m>1$, then $C$ is the only component of $H^{m}$ of a non-trivial linking number with all the other components, so we necessarily have that $F_v(C)=\pm C$. Similarly, for $m=1$ we get that both components are invariant under $F_v$, since at least one of them is. Since $C$ is fixed by $F_v$, it is not indexed by $\mathcal{E}(v)$, i.e., $C$ is a component of $L(v)_A$. We separately consider the cases $F_v(C)=C$ or $F_v(C)=-C$. In each case, we obtain some constraints on $p, q$, or $k$. Then we provide an amphichiral involution in the possible case. Moreover, when $p=1$, we also provide such an involution with exactly two fixed points.

Suppose that $F_v(C)=+C$. In this case, we will show that $L(v)$ is the given link $L$ itself, and provide the desired involutions for $p\ge 2$ and $p=1$. Note that $q\ge 1$ and we reindex $L(v)_{A^+}$ as $C=K_1^+$. Since any other component of $L(v)$ has non-zero linking number with $C$, and $C$ is a component of the unlink $L(v)_{A^+\cup \mathcal{E}(v)}$, we have $$|A^+\cup \mathcal{E}(v)|=q+k = 1.$$ Hence, $k=\mathcal{E}(v)=0$ so that $\G_L=\{v\}$. Then $L=L(v)$ itself is the key-chain link $H^m$.

Now, for any $p\ge 1$, we can take the symmetry $g$ to be the reflection around the plane of the diagram in Figure~\ref{fig:key-chain}, which in particular contains the core $C$. Note that $\Fix(g)\cong S^2$. When $p=1$, our $L=L(v)$ is the Hopf link with a $(1, 1)$-amphichiral map. In this case, we take another $(1,1)$-amphichiral involution $g$ as depicted in Figure~\ref{fig:Hopf-link-symmetry}, with $\Fix(g)\cong S^0$.

\begin{figure}[h]
  \includesvg{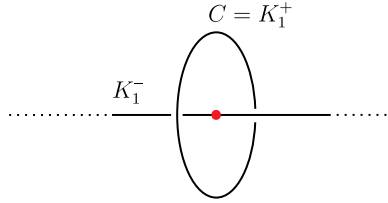}
  \caption{Note that in this Hopf link case, we may regard $C$ as a ring. The symmetry $g$ for the $(1,1)$-amphichiral Hopf link is given by the composition of the $\pi$-rotation around the central red point and the reflection around the plane of the diagram.}
    \label{fig:Hopf-link-symmetry}
\end{figure}

Now it remains to consider the case when $F_v(C) = -C$. In this case, we first show that the only possible $p$ is $1$, and we find an involution $G_v$ of $L(v)$. Then by using $G_v$ and the given amphichiral map $f$, we construct the desired amphichiral involution $g$ with $\Fix(g)\cong S^0$.

We reindex $L(v)_{A^-}$ as $C=K_1^-$. If $p>1$, then $K_2^-$ would be one of the rings of $H^m$ and we would have $$\lk(K_1^-,K_2^-)=-\lk(F_v(K_1^-),F_v(K_2^-))=-\lk(-K_1^-,-K_2^-)=-\lk(K_1^-,K_2^-),$$ and hence $\lk(K_1^-,K_2^-)=0$, which contradicts with $K_1^-=C$ being the core and $K_2^-$ being a ring. Therefore, $p=1$.

Observe now that we can find an orientation-reversing involution $G_v$ of $(S^3, L(v))$ with $\Fix(G_v)\cong S^0$, as illustrated in Figure~\ref{fig:keychain_symmetry}, such that 
\begin{itemize}
    \item $G_v(K_1^-)=-K_1^-$,
    \item $G_v(K_i^+)=+K_i^+$ for $i=1,\dots,q$, 
    \item $G_v(K_i^e)= K_{i+k/2}^e$ and $G_v(K_{i+k/2}^e)= K_{i}^e$ for $1\leq i\leq k/2$.
\end{itemize}

\begin{figure}[h]
  \includesvg{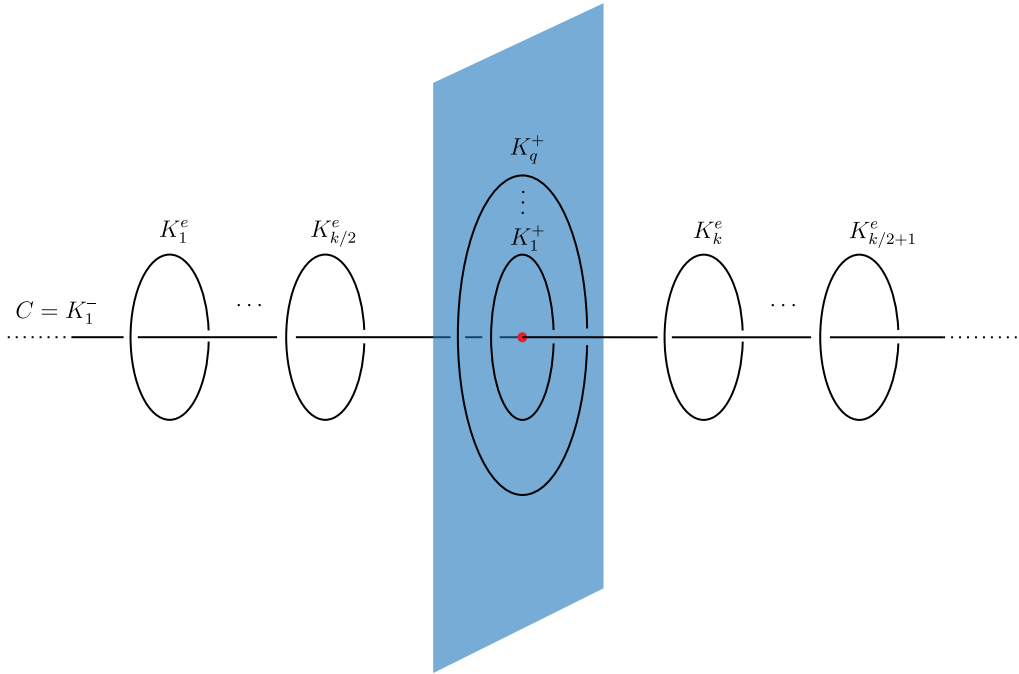}
  \caption{The symmetry $G_v$ is similarly defined with the map $g$ in Figure~\ref{fig:Hopf-link-symmetry}, but is used when $F_v(C)=-C$. The sublink $L(v)_{A^-}$ lies in the shaded plane which is vertical to $C=K_1^-=L(v)_{A^+}$. This $G_v$ fixes $K_1^-$, $K_1^+,\ldots, K_q^+$, and swaps $K_i^e$ and $K_{i+k/2}^e$.}
    \label{fig:keychain_symmetry}
\end{figure}

Now consider the edge-cuts of $\G_L$ along $\mathcal{E}(v)$ as $$\G_L-v= \bigsqcup_{i=1}^{k} \G_i,$$ where the connected components are labeled by $\G_i$ so that $M(\G_i)$ meets $v$ in the boundary component corresponding to $K_i^e$. For simplicity, write $M(\G_i)=Y_i$.

Let $i$ be an integer such that $1\le i \le k/2$. Recall that since $F_{v}(K_i^e)=K^e_{i+k/2}$, we have $f(Y_{i})=Y_{i+k/2}$.
Since $G_v$ acts in the same fashion as $F_v$ on $K_i^e$, we have that their restrictions to $\partial Y_i$ are isotopic.
Therefore, for $1\le i \le k/2$, we can isotope $G_v$ in a collar neighborhood of $\partial Y_{i}\subset Y_{i}$ so that it agrees with the restriction of $F_v$ to $\partial Y_i$. Similarly, since $G_v$ acts in the same way as $F_v^{-1}$ on $K_{i+k/2}^e$, we can isotope $G_v$ to make it agree with $F_v^{-1}$ on $\partial Y_{i+k/2}$.

Define now the map $g:(S^3,L)\to (S^3,L)$ as follows: for each $i$ such that $1\leq i \leq k/2$,
\begin{itemize}
    \item $g|_v=G_v:v\to v$,
    \item $g|_{Y_i}=f|_{Y_i} :Y_i\to Y_{i+k/2}$,
    \item $g|_{Y_{i+k/2}}=(f|_{Y_i})^{-1}:Y_{i+k/2}\to Y_i$.
\end{itemize} 
See Figure~\ref{fig:involution-construction} for a schematic picture of an example. By construction, $g$ is an orientation-reversing involution with $\Fix(g)=\Fix(G_v)\cong S^0$ on $(S^3,L)$, so $(L,g)$ is strongly $(p,q)$-amphichiral.

\begin{figure}[h]
  \includesvg{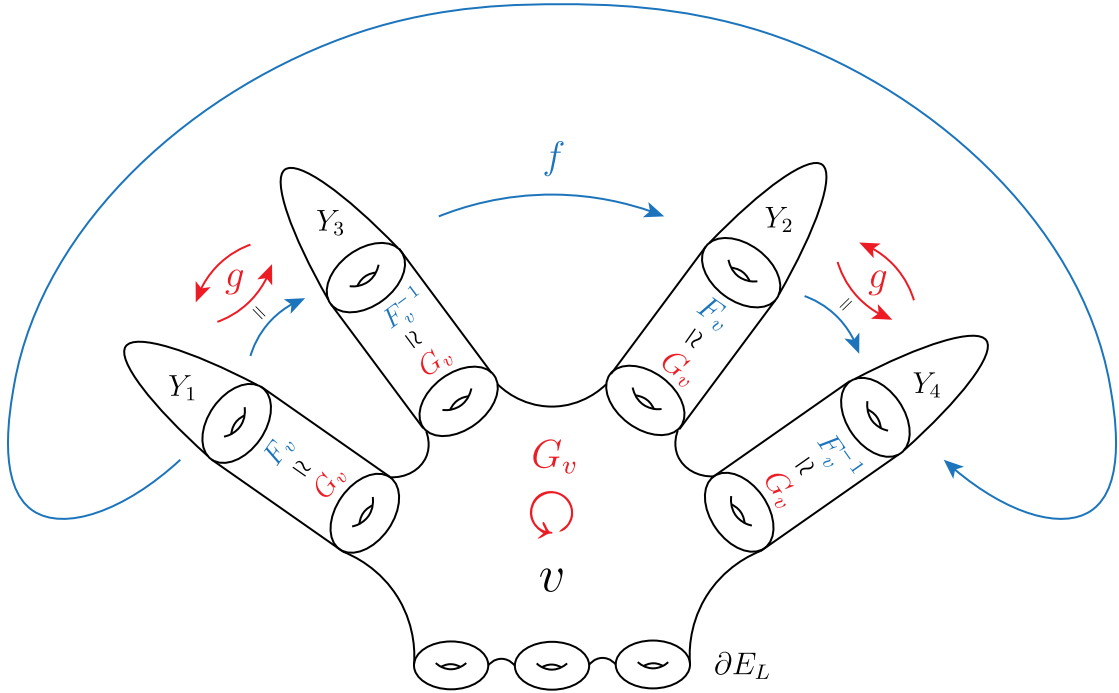}
  \caption{A schematic figure of the construction of the amphichiral involution $g$ for $L$ using the amphichiral map $f$ for $L$ and the involution $G_v$ on $v$ for $L(v)$. Note that $G_v$ is isotopic to $F_v$ and $F_v^{-1}$ on $\partial Y_i$ and $\partial Y_{i+k/2}$ for $1\le i \le k/2$, respectively, so that we can isotope $G_v$ through a collar neighborhood of $\partial Y_i$ to make it agree with $F_v$ or $F_v^{-1}$ on $\partial v$. Then we can glue $g|_{Y_i}$ to $G_v$ to obtain the amphichiral involution $g$ on $(S^3, L)$.}
    \label{fig:involution-construction}
\end{figure}

\vspace{0.25em}
\noindent \textbf{Case 2.} $v$ is hyperbolic.
\vspace{0.25em}

In this case, up to isotopy, we have that the restriction $f|_{v}$ is a hyperbolic isometry by Mostow’s rigidity theorem (for example, see \cite[Theorem~13.3.6]{Mar23}), and hence has finite order by \cite[Corollary~4.3.8]{Mar23}. Since $f$ is reduced, the order of $f|_{v}$, and hence, the order of $F_v$ is $2^m$ with $m\geq 1$. First, we consider the case where $m=1$. In this case, we will construct our involution in a similar way to \textbf{Case 1.b}, but without changing the cardinality of the orbits since $F_v$ is already order $2$. After that, we claim that the case where $m > 1$ cannot happen.

\vspace{0.25em}
\noindent $\blacktriangleright$ \textbf{Case 2.a.} $m=1$ and $p\geq 2$.
\vspace{0.25em}

Since $F_v$ is already an order $2$ map, we can take an orientation-reversing involution $G_v$ of $(S^3, L(v))$ just as $F_v$. Since $F_v^{-1}=F_v$, we can proceed in the same way as in \textbf{Case 1.b} to define an amphichiral involution $g:(S^3,L)\to (S^3,L)$.

\vspace{0.25em}
\noindent $\blacktriangleright$ \textbf{Case 2.b.} $m=1$ and $p=1$.
\vspace{0.25em}

As \textbf{Case 2.a}, since $F_v$ is already of order $2$, we can obtain the desired involution $g:(S^3,L)\rightarrow (S^3, L)$ using $F_v$. Furthermore, in this case, we prove that $\Fix(g)\cong S^0$.

Recall from Theorem~\ref{thm:smith-orthogonal} (2) that $\Fix(F_v)$ is diffeomorphic to either $S^2$ or $S^0$. We first argue that the former case is impossible. Suppose for a contradiction that $\Fix(F_v)\cong S^2$. Let $S^3 - \Fix (F_v) = B_1 \cup B_2$. 

It is clear from the assumption $p=1$ that $$|L(v)_{A^-}\cap \Fix(F_v)|=2.$$ Since $F_v(B_1)=B_2$ and $F_v$ preserves the orientation of $L(v)_{A^+}$, we have that $$L(v)_{A^+}\subset \Fix(F_v).$$ Moreover, since $F_v$ permutes every component of $L(v)_{\mathcal{E}(v)}$, we see that $$L(v)_{\mathcal{E}(v)}\cap \Fix(F_v)=\varnothing.$$

Now, since $\Fix (F_v) \cap E_{L(v)}$ is obtained by removing $q$ disjoint circles from a two-punctured sphere, by an Euler characteristic argument, it is not difficult to see that $\Fix (F_v) \cap E_{L(v)}$ contains either an essential annulus or a disk with the boundary given by a longitude of $K_i^-$ for some $1\leq i\leq q$ in $E_{L(v)}$, see \cite[Sections~9.3.1 and 9.4.7]{Mar23}. The latter case shows that the component $K_i^-$ is a split unknotted component; in particular, $v$ is not even irreducible. In either case, $v$ is not hyperbolic by \cite[Proposition~9.4.18]{Mar23}, which is a contradiction. Therefore, $\Fix(F_v)\cong S^0$. Thus, the same holds for $g$ by construction.

\vspace{0.25em}
\noindent $\blacktriangleright$ \textbf{Case 2.c.} $m\ge 2$ and $p\ge 2$.
\vspace{0.25em}

Recall that an orientation-reversing map on $S^1$ of finite order always has order $2$ and has exactly two fixed points. Since we assume that $p\geq 2$, the finite order map $F_v:(S^3,L(v))\to (S^3,L(v))$ would have at least $2p\ge 4$ fixed points: two on each $K_i^-$. By Theorem~\ref{thm:smith-orthogonal} (2), $\Fix(F_v)$ is diffeomorphic to $S^0$ or $S^2$. 
Since $\Fix(F_v)$ contains more than two points, the fixed point set would necessarily be a $2$-sphere, and hence the order $2^m$ of $F_v$ would be equal to $2$, which is a contradiction.

\vspace{0.25em}
\noindent $\blacktriangleright$ \textbf{Case 2.d.} $m\ge 2$ and $p = 1$.
\vspace{0.25em}

While we know $F_v$ is of finite order, $f$ may not be of finite order. We will construct a map $g$ of the same order as $F_v$ in a similar way to \textbf{Case 1.b}. While we have previously obtained an obvious involution $G_v$ and extended it to obtain $g$, we instead extend $F_v$ to get a finite order map $g$. Using this $g$, however, we show that the only possible $L$ with such a finite order amphichiral map $g$ is the unknot, which contradicts that $v$ is hyperbolic.

Notice that $F_v^2|_{K_1^-}=\id_{K_1^-}$ since $F_v|_{K_1^-}$ is an orientation-reversing map on $S^1$ of finite order. By Theorem~\ref{thm:smith-orthogonal} (1), $F_v^2$ is then a rotation of order $2^{m-1}$ along the \emph{unknotted} axis $K_1^-$. Since $L(v)$ is hyperbolic, by applying Lemma~\ref{lemma:periodic_unlink}, we get that the orbit of every component of $L(v)$ has cardinality $2^{m-1}$ under the action of $F_v^2$, and hence cardinality $2^{m}$ under the action of $F_v$. In particular, this implies that $K_1^-$ is the only component of $L(v)$ invariant under the action of $F_v$, i.e., $p=1$ and $q=0$.

Now we construct a finite order map $g:(S^3,L)\to (S^3,L)$ in a similar fashion to \textbf{Case 1.b}. Let $l$ be the number of orbits, namely $l=|\mathcal{E}(v)| / 2^m$. Let $\G_1,\dots, \G_l$ be representatives of the orbits of the action of $f$ on the connected components of $\G_L-v$, i.e., $$\G_L-v= \underset{i=1}{\overset{n}{\sqcup}} \operatorname{Orb}(\G_i), \quad \text{where} \ \operatorname{Orb}(\G_i)=\G_i\sqcup f(\G_i)\sqcup \dots \sqcup f^{2^m-1}(\G_i).$$

For simplicity, let $Z_i$ denote the corresponding $3$-manifold $M(\G_i)$. Define the map $g:(S^3,L)\to (S^3,L)$ as follows: for each $i$ such that $1\le i \le l$,
\begin{itemize}
    \item $g|_v=F_v:v\to v$,
    \item $g|_{f^j(Z_i)}=f|_{f^j(Z_i)} :f^j(Z_i)\to f^{j+1}(Z_i)$ for $0\leq j\leq 2^m-2$,
    \item $g|_{f^{2^m-1}(Z_i)}=(f^{2^m-1}|_{Z_i})^{-1}:f^{2^m-1}(Z_i)\to Z_i$.
\end{itemize} 

In contrast to the previous case, since we directly use $F_v$ obtained from $f$, it is not necessary to isotope our $g$ near the boundary components to make it agree with $f$ there. Without isotopy, observe that $g$ agrees with $F_v$ on $\partial (f^j(Z_i))$ for all $i$ and $j < 2^m-1$, and similarly with $(F_v^{2^m-1})^{-1}$ on $\partial (f^{2^m-1}(Z_i))$ by definition.

Now we have an orientation-reversing map $g$ of $(S^3, L)$ of order $2^m$ and get the orientation-preserving map $g^2$ of order $2^{m-1}> 1$. Then, by Theorem~\ref{thm:smith-orthogonal} (1), since $L\subset \Fix(g^2)\subsetneq S^3$, $L$ is the unknot, contradicting the assumption that $v$ is hyperbolic.
\end{proof}

Recall that if $T$ is fixed by $f$, then its end vertices $v$ and $w$ are also fixed as shown in Proposition~\ref{prop:graph_properties} (4). Then the induced maps $F_v$ and $F_w$ fix the components $L(v)_T$ and $L(w)_T$ of the corresponding links $L(v)$ and $L(w)$, respectively, indexed by this fixed edge $T\in \mathcal{E}(v)\cap \mathcal{E}(w)$. In this case, we can ask whether $F_v$ and $F_w$ preserve or reverse the orientation of the components $L(v)_T$ and $L(w)_T$, respectively. We will use $\epsilon_T^v,\epsilon_T^w\in\{\pm 1\}$ to denote it. The following lemma, which concerns this situation, can be easily proved by following \cite[Theorem~4.1]{Har80} and \cite[p. 466]{KW18}. For the convenience of the readers, we give a proof based on Definition~\ref{defn:companionship-graph}. 

\begin{lem}\label{lemma:extension}
  Let $(L,f)$ be an amphichiral link. If an edge $T$ of $\G_L$ is fixed by $f$, then $$F_v (L(v)_T) = \epsilon_T^v L(v)_T \quad \text{and} \quad F_w (L(w)_T) = \epsilon_T^w L(w)_T,$$ where $\epsilon_T^v \cdot \epsilon_T^w = -1$.
\end{lem}
\begin{proof}

  Recall that we oriented $L(v)_T$ and $L(w)_T$ so that the preimages $\lambda_v$ and $\lambda_w$ of their preferred longitudes under the untwisted re-embeddings have the linking number $1$ after perturbing one into the other side. This is equivalent to that we set them as the preimages on the torus $T$ have the signed intersection number $1$, namely, $$\lambda_v \cdot \lambda_w = 1 \text{ in } T.$$ Since $f$ reverses the orientation of $E_L$ and fixes both $v$ and $w$, $f$ should reverse the orientation of the torus $T$. If $F_v$ and $F_w$ both preserved the orientations of $L(v)_T$ and $L(w)_T$ respectively, then $$f(\lambda_v) \cdot f(\lambda_w) = -1 \text{ in }f(T),$$ which is against the definition. Therefore, exactly one of them has to reverse the orientation.
\end{proof}

Consider a satellite knot $P(K)$ and its companionship graph $\G_{P(K)}$. Let $L$ be the link consisting of the pattern $P$ and the meridian $U$ of the solid torus in which $P$ lies. Suppose that each of $E_K$ and $E_L = S^1\times D^2 - P$ has a single JSJ piece. Then $\G_{P(K)}$ consists of two vertices: $v$ labeled by $L$ and $w$ labeled by $K$.

Now suppose that $P(K)$ is negative amphichiral with respect to a map $f$. Then $f$ fixes the unique edge $T$, directed from $w$ to $v$. Then, by Lemma~\ref{lemma:extension} above, the following are equivalent:
\begin{itemize}
  \item $\epsilon_T^v = 1$,
  \item $\epsilon_T^w = -1$,
  \item $(K, F_w)$ is negative amphichiral.
\end{itemize}

In the sense that the companion $K$ of the negative amphichiral satellite $P(K)$ is also negative amphichiral when $\epsilon_T^v=1$, we will call this directed fixed edge $T$ \textit{coherently directed with respect to $f$}, and define it generally in the link case. The edge-cut along a coherently directed edge is described in Figure~\ref{fig:lemma5.3}.

\begin{defn}
  Let $(L, f)$ be a negative amphichiral link. For a fixed edge $T$ in $\G_L$, we say $T$ is \textit{coherently directed with respect to $f$} if $T$ is directed from $w$ to $v$, and $f$ reverses the orientation of the component $L(w)_T$ of $L(w)$, i.e., $$\epsilon_T^v=1.$$
\end{defn}

\begin{figure}[htbp]
\centering
\includesvg[width=0.7\columnwidth]{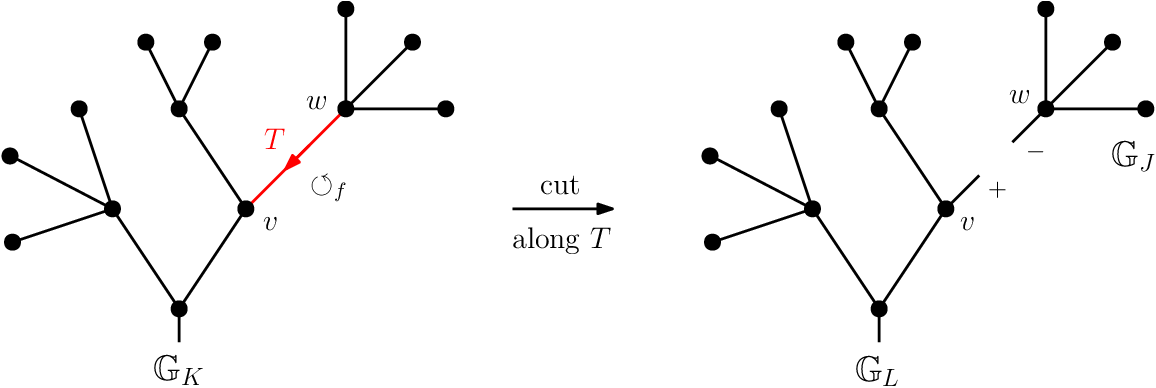}
\caption{The edge-cut along the coherently directed edge $T$ in $\G_K$ for a negative amphichiral $K$ into $\G_L$ and $\G_J$, where $L=P\cup U$ for $U$ the component corresponding to $T$. Then $K$ is the splice of $J$ and $L$, and in particular, $K=P(J)$. Since $T$ is coherently directed, the companion $J$ is again negative amphichiral.}
\label{fig:lemma5.3}
\end{figure} 

\noindent We drop \say{with respect to $f$} when it is clear in the context. We also say $T$ is \textit{incoherently directed} if $\epsilon_T^v = -1$. In the proof of Theorem~\ref{thm:link}, when we consider the case where $f$ fixes an edge, we divide it into subcases based on whether there is a coherently directed edge or not, and separately prove them.

For the remaining part of this section, we will only consider the case when $L$ is a negative amphichiral knot $(K,f)$. Recall that $\G_K$ is a \textit{rooted} tree when $K$ is a knot. Since every edge is directed in this case, a fixed edge is either coherently or incoherently directed, but not undirected. We define several notions on $\G_K$ combined with the action of $f$. These notions appear in the statements of Theorems~\ref{thm:SNACK} and \ref{thm:SNACK-concordance}, which will be proved in Section~\ref{sec:5}.

\begin{defn}
  Let $(K,f)$ be a negative amphichiral knot. Let $v$ be a vertex fixed by $f$. We say that $v$ is \emph{coherent} if every fixed edge $T$ incident to $v$ is coherently directed, i.e.,
\begin{itemize}
    \item $\epsilon_T^v=+1$ if $T$ is entering $v$,
    \item $\epsilon_T^v=-1$ if $T$ is exiting $v$.
\end{itemize}
\noindent Similarly, we say that $v$ is \emph{incoherent} if $v$ has at least one incident fixed edge, and every such edge is incoherently directed.
\end{defn}

Note that a vertex without an incident fixed edge is coherent. In addition, a vertex that has both incident coherently directed edges and incoherently directed edges is neither coherent nor incoherent. 

\begin{defn}
\label{defn:JSJstructure}
Let $(K,f)$ be a negative amphichiral knot. Let $\G_{max}$ be the unique maximal subtree of $\G_K$ taken from subtrees $\G$ which satisfy both of the following:
\begin{itemize}
    \item $\G$ contains the root,
    \item Every fixed vertex in $\G$ is coherent.
\end{itemize}
We call this $\G_{max}$ the \textit{maximal coherent subtree of} $\G_K$. Using $\G_{max}$, we define two notions.

\begin{itemize}
    \item $(K,f)$ has a \emph{totally coherent JSJ structure} if every vertex in $\G_K$ fixed by $f$ is coherent, i.e., $$\G_{max}=\G_K.$$
    \item $(K,f)$ has a \emph{properly incoherent JSJ structure} if the root of each $\G_i$ is incoherent root, where $$\G_K-\G_{max} = \bigsqcup_{i=1}^n \G_i.$$
\end{itemize}
When the context is clear, we simply call $\G_K$ totally coherent or properly incoherent. See Figure~\ref{fig:max-special} for an example.
\end{defn}

\begin{figure}[h]%
    \centering
    \includesvg[scale=1.5]{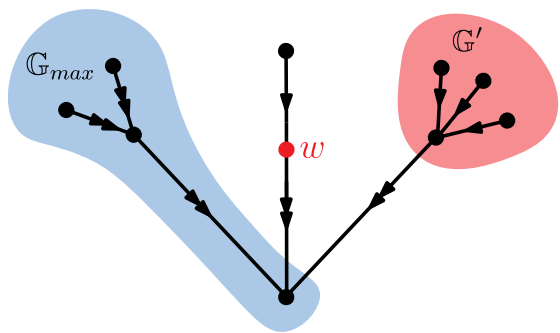}
    \caption{A companionship graph $\G_K$ of a negative amphichiral knot $K$. For simplicity, assume that every edge is fixed. We use double-headed arrows to denote coherently directed edges. The maximal coherent subtree $\G_{max}$ lies in the blue-shaded area. The root of the subgraph $\G'$ in the red-shaded area, as the companionship graph of another negative amphichiral knot, is incoherent. On the other hand, notice that the red-colored vertex $w$ is fixed, but neither coherent nor incoherent. In particular, $\G_{max}$ does not contain $w$. Notice that the whole graph $\G_K$ is properly incoherent.}%
    \label{fig:max-special}%
\end{figure}

Note that each of the above subgraphs $\G_i$ is the companionship graph of a nontrivial knot by Proposition~\ref{prop:many-edge-cut}, so the properly incoherent condition is well-defined. Now we provide several properties directly obtained from definitions.

\begin{prop}
\label{prop:coherent-properties}
Let $(K,f)$ be a negative amphichiral knot. Then the following hold:
\begin{enumerate}
    \item If $\G_K$ has no fixed edge, then $\G_K$ is totally coherent.
    \item $f(\G_{max})=\G_{max}$.
    \item If the root is not coherent, then $\G_{max}=\varnothing$.
    \item If $\G_{max}=\varnothing$, then $\G_K$ is properly incoherent if and only if the root of $\G_K$ is incoherent.
\end{enumerate}
\end{prop}


\section{Proof of Theorem~\ref{thm:link}}
\label{sec:4}

Let $(L, f)$ be a negative amphichiral link. In this section, we first recall the notions of concordance and sliceness of links in $S^3$ in a generalized setting. Then we obtain concordance results of spliced links in that setting, provided a concordance of the left-hand side, or the right-hand side of the bow tie notation, respectively. Then we consider the concordance of strongly negative amphichiral links with more than one component. Finally, we prove Theorem~\ref{thm:link}. Our argument is based on the action of $f$ on the companionship graph $\G_L$ described in Section~\ref{sec:3}.

\begin{defn}\label{def:concordance}
Let $W^4$ be a closed connected $4$-manifold, and let $L=K_1\cup\dots\cup K_n$ and $L'=K'_1\cup\dots\cup K'_n$  be $n$-component links in $S^3$. Let $B_0$ and $B_1$ be disjoint small $4$-balls in $W$. We say that $L$ and $L'$ are \emph{concordant in $W$} if there exists disjoint cylinders $C_i$ smoothly and properly embedded in $W-(B_0\cup B_1)$ for $i=1, \ldots, n$ such that 
\begin{itemize}
    \item each $C_i$ is diffeomorphic to $S^1 \times [0,1]$,
    \item $\partial (W-(B_0\cup B_1),C_i)= -(\partial B_0,K_i) \cup (\partial B_1, K_i')$.
\end{itemize}
When the cylinders $\mathcal{C}=C_1\cup\cdots \cup C_n$ are specified, we call $\mathcal{C}$ a \textit{concordance between $L$ and $L'$ in $W$}, and we say $L$ and $L'$ are concordant in $W$ \textit{through $\mathcal{C}$}. We also say that $L$ and $L'$ are \emph{concordant} (resp. \emph{rationally concordant}) if they are concordant in $S^4$ (resp. in a $\Q HS^4$).
\end{defn}

\begin{defn}
An $n$-component link $L$ in $S^3$ is said to be \emph{slice in $W^4$} if $L\subset S^3=\partial (W-B^4)$ bounds disjoint $n$ disks smoothly and properly embedded in $W-B^4$. Equivalently, $L$ is slice in $W$ if and only if $L$ is concordant to the $n$-component unlink in $W$.
We say that $L$ is \emph{slice} (resp. \emph{rationally slice}) if it is slice in $S^4$ (resp. in a $\Q HS^4$).
\end{defn}

\begin{remark}
If a $4$-manifold $W$ has $\partial W=S^3$, we also say that a link $L\subset S^3$ is slice in $W$ if it is slice in $W\cup_{\partial W} B^4$.
\end{remark}

Consider a satellite knot $P(K)$ with the companion $K$ and the pattern $P$. Suppose $K$ is concordant to another knot $K'$ through the cylinder $C\cong S^1\times I$. Regarding $S^3$ as $S^1\times D^2\cup (S^3-K)$ and $P(K)$ as the image of $P\subset S^1\times D^2$, we obtain a concordance $P(C)$ between $P(K)$ and $P(K')$ as the image of $P\times I \subset S^1\times D^2 \times I$ by viewing $S^3\times I$ as $ S^1\times D^2\times I\cup (S^3\times I - C)$.

Recall that the satellite operation is a special case of the splicing operation. For example, we can view $P(K)$ as $K\bowtie (P\cup U)$ where $U$ is the meridian of the solid torus in which $P$ lies. Then $K\bowtie (P\cup U)$ is concordant to $K'\bowtie (P\cup U)$. In the lemma below, we prove a similar result for the splicing operation in a general $4$-manifold. We also consider the case where a concordance on the right-hand side of the bow tie notation is provided under more specific conditions. This lemma will be essentially used to get (rational) concordances of spliced links.

\begin{lem}\label{lemma:splice_concordance}
Let $(L_1,\dots,L_n)\underset{K_i\sim U_i}{\bowtie}L$ be the spliced link. Then the following hold:
\begin{enumerate}
  \item If $L_i$ is concordant to $L_i'$ in $W_i^4$ through $\mathcal{C}_i$ such that $C_i\cap L_i = K_i$ for each cylinder $C_i\subset\mathcal{C}_i$, then $$(L_1,\dots,L_n)\underset{K_i\sim U_i}{\bowtie}L \text{ is concordant to }(L_1',\dots,L_n')\underset{K_i'\sim U_i}{\bowtie}L\text{ in } \underset{i=1}{\overset{n}{\#}} W_i,$$ where $K_i'= C_i \cap L_i'$.
  \item If $L$ and $L'$ are concordant through $\mathcal{C}=C_1\cup\cdots \cup C_m$ with $C_i\cap L = U_i$ for $i=1,\ldots,n\le m$ such that $$S^3\times I-(C_1\cup\dots\cup C_n)\cong (S^3-(U_1\cup\dots\cup U_n))\times I,$$ then $$(L_1,\dots,L_n)\underset{K_i\sim U_i}{\bowtie}L\text{ is concordant to }(L_1,\dots,L_n)\underset{K_i\sim U_i'}{\bowtie}L',$$ where $U_i' = C_i\cap L'$.
\end{enumerate}
\end{lem}

\begin{proof}
    (1) Note that the sublink $L-(U_1\cup \cdots \cup U_n)$ is contained in $\#^n (S^1\times D^2)\cong S^3 - (U_1\cup \cdots \cup U_n)$. Let $Z_i$ be two punctured $W_i$, namely $Z_i = W_i - (B_0 \cup B_1)$, and consider the complement of $C_i$ in $Z_i$. Note that the component $C_i$ of $\mathcal{C}_i$ is a concordance between $K_i$ and $K_i'$ in $W_i$.

  The gluing maps $\varphi_i$ and $\varphi_i'$ to obtain the splices extend to the gluing map $\psi_i$, which levelwisely glue $Z_i - C_i$ to $S^3-(U_1\cup\cdots\cup U_n)$ along $T^2\times I$. Note that $$((S^3-U_i)\times I)\cup_{\psi_i}(Z_i-C_i)\cong Z_i,$$ since it is just gluing $\nu(C_i)\cong S^1\times D^2\times I$ back in $Z_i$. Let $\Psi$ be the disjoint union of $\psi_i$. Now, we construct the ambient manifold $Z$ as: \begin{align*}Z &= \left( (S^3 - (U_1\cup \cdots \cup U_n))\times I \right) \cup_{\Psi} \left( \sqcup_i (Z_i-C_i) \right)\\
  &\cong \left(  (\#^n (S^1\times D^2))\times I \right) \cup_{\Psi} \left( \sqcup_i(Z_i-C_i) \right).\end{align*}
  
\noindent Note that $Z$ is the resulting manifold obtained by \textit{levelwise connected sum} of $Z_i$ along $B^3\times I\subset (S^1\times D^2)\times I$. More precisely, letting $S_{i,r}$ and $S_{i+1,l}$ be the $2$-spheres identified when $S^3-U_i$ and $S^3-U_{i+1}$ are joined by connected sum, the levelwise connected sum identifies $S_{i,r}\times I$ and $S_{i+1,l}\times I$ levelwisely.
  
  Now we cap $Z$ off by two $4$-balls $B^+$ and $B^-$. Regard $B^{\pm}$ as the boundary sums $\natural_i B^{\pm}_i$ and the identified $3$-balls $D^{\pm}_{i,r}$ and $D^{\pm}_{i+1,l}$ as the caps of $S_{i,r}\times I$ and $S_{i+1,l}\times I$, respectively. Then the gluings are done along $( D^+_{i,r}\cup S_{i,r} ) \times ( I \cup D^-_{i,r} )\cong S^3$ and $(D^+_{i+1,l}\cup S_{i+1,l} )\times (I\cup D^-_{i+1,l})\cong S^3$. Thus, the capped off manifold is $$Z\cup B^- \cup B^+ = \underset{i=1}{\overset{n}{\#}} W_i.$$ Therefore, the cylinders $(L-(U_1\cup \cdots \cup U_n))\times I$ in $(S^3-(U_1\cup\cdots\cup U_n))\times I$ and $\mathcal{C}_i - C_i$ in $Z_i - C_i$ form the desired concordance in $\#_{i=1}^n W_i$.

  (2) We form the ambient manifold as $$\left( (S^3 - (U_1\cup \cdots \cup U_n))\times I \right) \cup_{\Psi} \left( \sqcup_i ((S^3-K_i)\times I) \right) \cong S^3\times I.$$ Since $(S^3-(U_1\cup\cdots\cup U_n))\times I\cong S^3 \times I - (C_1\cup \cdots \cup C_n)$, the cylinders $\mathcal{C}-(C_1\cup\cdots \cup C_n)$ and $(L_i-K_i)\times I$ form the desired concordance.
\end{proof}

Now we provide two lemmas on concordance of strongly negative amphichiral links. While every strongly negative amphichiral knot is rationally slice \cite{Kaw09}, it is not slice in $B^4$ in general. On the other hand, we prove that every strongly negative amphichiral \emph{link} with more than one component is slice in $B^4$.

\begin{lem}\label{lemma:str}
  Let $(L,f)$ be a strongly negative amphichiral link. If $|L|>1$, then $L$ is slice.
\end{lem}

\begin{proof}

By Theorem~\ref{thm:smith-hom-sphere}, the restriction of $f$ to each component of $L$ has two fixed points. Since $$\Fix(f) = S^0 \ \text{or} \ S^2 \quad \text{and} \quad \Fix(f) \supset \bigcup_{i=1}^{|L|} \Fix(f|_{L_i}) \supsetneq S^0 ,$$ $\mathrm{Fix}(f)$ is always a $2$-sphere whenever $|L|>1$.

This $2$-sphere $\Fix(f)$ splits $L$ into the union of two tangles $T_1$ and $T_2$ with $2|L|$ endpoints. For example, see Figure~\ref{fig:str-neg-link}. Since $f(T_1)=T_2=\overline{T_1}$, regarding $(S^3, L) = \partial ((B^3, T_1)\times I)$, it is easily see that $L$ is a slice link in $B^4$.
\end{proof}

\begin{figure}[h]
  \includesvg{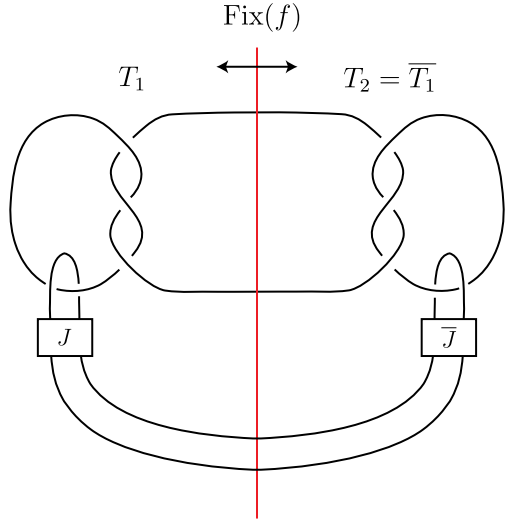}
  \caption{A strongly negative amphichiral link with two components.}
\label{fig:str-neg-link}
\end{figure}

\noindent Combined with \cite{Kaw09}, every strongly negative amphichiral link is rationally slice. In the proof of Theorem~\ref{thm:link}, we prove that every negative amphichiral link in the base case of our induction is strongly negative amphichiral. The lemma above completes the proof of the case.

Furthermore, we consider a more specific concordance when the given strongly amphichiral link has an unlinked sublink. This concordance is particularly necessary to apply Lemma~\ref{lemma:splice_concordance} (2).

\begin{lem}\label{lemma:unknotted_concordance}
Let $(L,f)$ be a strongly negative amphichiral link indexed by $A_1\sqcup A_2$. If each $A_i\neq \varnothing$ and the sublink $L_{A_2}$ is unlinked, then $L_{A_1}$ is concordant to a local unlink in $(S^3-L_{A_2})\times I$.
\end{lem}
\begin{proof}
As in the proof of Lemma~\ref{lemma:str}, since $|L|\geq 2$, the fixed point set of $f$ is a $2$-sphere $S$. The link $L$ is split into two tangles $T$ and its mirror $\overline{T}=f(T)$ by $S$, and each component of $L$ intersects $S$ in two points. For example, see Figure~\ref{fig:Tangle}. Notice that there are no crossings of $L$ on $S$.

\begin{figure}[h]
  \includesvg[scale=0.7]{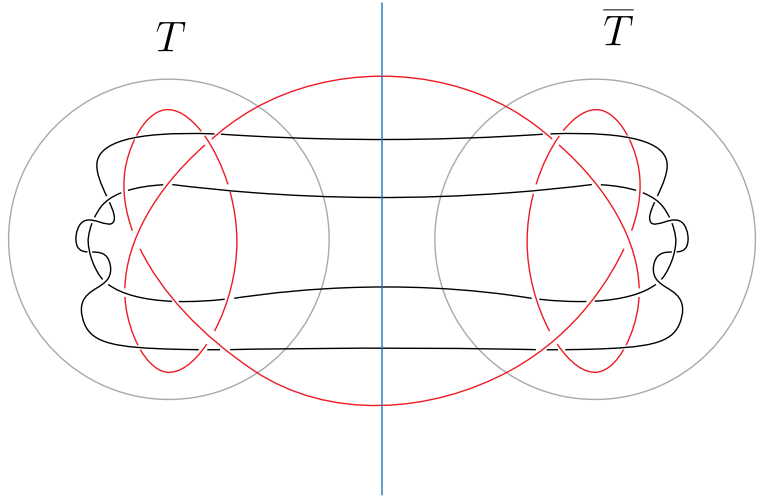}
  \caption{An example of $L$ with $p=1$ and $q=2$. The unlinked sublink $L_{A_2}$ is black-colored and the other sublink $L_{A_1}$ is red-colored. The axis for reflection is blue-colored.}
\label{fig:Tangle}
\end{figure}

Regard $S^3$ as the union of two $3$-balls $B_1^3$ and $B_2^3$ glued along the $2$-sphere $S$ such that $L\cap B_1 = T$ and $L\cap B_2 = \overline{T}$. Let $T_2$ be the part $L_{A_2}\cap T$ of $L_{A_2}$ in $T$. We isotope the tangle $T$ to the other tangle $T'$ in $B_1$ such that $T_2$ becomes untangled. While the usual tangle isotopy is an isotopy of $3$-ball relative to the boundary, our isotopy allows the boundary to be moved as well.

Then it is clear that there exists an isotopy $g_t$ of $B_1$ such that $T_2'=g_1(T_2)$ in the new tangle $T'=g_1(T)$ is trivial in $B_1$. Moreover, $h_t=f|_{B_1} \circ g_t \circ f|_{B_2}$ also gives an isotopy of $B_2$ which makes $\overline{T_2'}=h_1(\overline{T_2})$ in the new tangle $\overline{T'}=h_1 (\overline{T})$ is trivial in $B_2$. Moreover, $g_t$ and $h_t$ agree on the boundary $S$. Thus, we obtain an isotopy $G_t$ on $S^3$ which is equivariant with respect to $f$ and makes the tangles $T$ and $\overline{T}$ to the tangles $T'$ and $\overline{T'}$ whose intersections with $L_{A_2}$ are untangled.

\begin{figure}[h]
  \includesvg[scale=0.8]{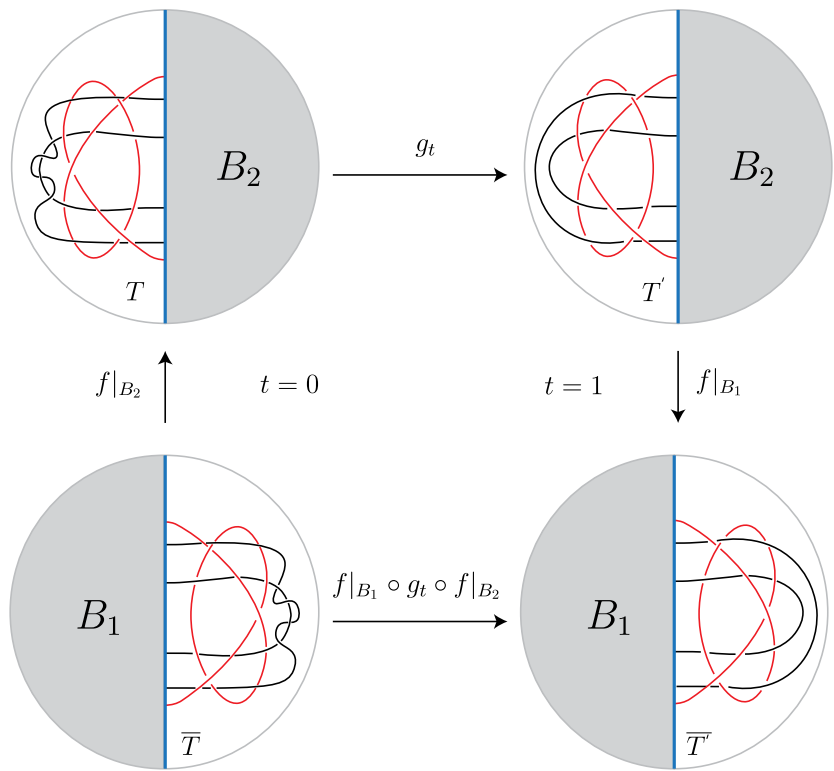}
  \caption{The isotopies $g_t$ defined on $B_1$ and $h_t = f|_{B_1}\circ g_t\circ f|_{B_2}$ on $B_2$ agree on the blue-colored reflection axis $S$ so that their glueing gives an isotopy $G_t$ on $S^3$ to make the part $T_2$ of $L_{A_2}$ in the tangle trivial in each of $B_1$ and $B_2$.}
\label{fig:equiv-isotopy}
\end{figure}

Let $p=|A_1|$ and $q=|A_2|$. Consider now the $(p+q)$ disks $D_1,\dots,D_{p+q}$ given by $T'\times I\subset B^3\times I\cong B^4$, where $D_1,\dots,D_p$ are bounded by the sublink $L_{A_1}$, and $D_{p+1}\cup \cdots \cup D_{p+q}$ are the standard trivial disks bounded by the unlinked sublink $L_{A_2}$. Pick a point $x$ in $B^4-(D_1\cup \cdots \cup D_{p+q})$, and join $x$ with a point in $\int(D_i)$ by disjoint paths $\gamma_i$ for $i=1,\ldots, p+q$ in $B^4$ such that $\int(\gamma_i)\cap D_j = \varnothing$ for $j\neq i$. Let $N$ be a regular neighborhood of $\cup_{i=1}^{p+q} \gamma_i$, which is diffeomorphic to a $4$-ball. 

Then $(D_1,\dots D_{p+q})\cap (B^4-N)$ forms a concordance between $L$ and the $(p+q)$-component unlink in $B^4-N\cong S^3\times I$. Since the disks $D_{p+1},\dots,D_{p+q}$ are trivial, the annuli $(D_{p+1},\dots D_{p+q})\cap (B^4-N)$ are isotopic to the trivial annuli $L_{A_2}\times I\subset S^3\times I$. Hence, $(D_1,\dots,D_p)\cap B^4-N$ is a concordance between $L_{A_1}$ and the $p$-component unlink in $(B^4-N)-(D_{p+1}\cup\dots\cup D_{p+q})\cong (S^3-L_{A_2})\times I$.
\end{proof}

The rest of this section is devoted to proving Theorem ~\ref{thm:link}. Since the proof involves several cases, we first provide a brief outline. Recall that $f$ acts on the companionship graph $\G_L$ as in Section~\ref{sec:3}. We consider two cases, depending on whether $f$ fixes an edge or not.

When there are no fixed edges, we will give a direct proof, based on Lemma~\ref{lemma:elementary-graph}. We pick the unique vertex $v$ containing $\partial E_L$ and check the condition in Lemma~\ref{lemma:elementary-graph} to prove that $L$ is strongly negative amphichiral, which implies that $L$ is rationally slice by Lemma~\ref{lemma:str}.

When $f$ fixes an edge $T$, we give the proofs separately according to whether $T$ is coherently directed, undirected, or every fixed edge is incoherently directed. In any case, we obtain a splice form of $L$ from some negative amphichiral links with less complexity. Based on induction on $c(L)$, we apply Lemma~\ref{lemma:splice_concordance} to get a concordance in a boundary sum of Kawauchi manifolds.

Now we restate Theorem \ref{thm:link}, with respect to specified rational balls obtained by taking the boundary sum of Kawauchi manifolds, based on \cite{Lev23}.
\begin{repintrothm}{thm:link}
    Every negative amphichiral link is slice in $\natural^n V$ for some $n\ge 0$, where $V$ is the Kawauchi manifold.
\end{repintrothm}

\begin{proof}[Proof of Theorem~\ref{thm:link}]
Let $(L,f)$ be a negative amphichiral link indexed by $A$. First, notice that if $L$ is a split link, its split sublinks are also negative amphichiral. Therefore, we may suppose that $L$ is non-split. Moreover, as mentioned in Remark~\ref{rmk:reduced}, we assume that $(L,f)$ is reduced. Consider the action of $f$ on the companionship graph $\G_L$ of $L$.


\vspace{0.25em}
\noindent \textbf{Case 1.} \emph{$f$ fixes no edges in $\G_L$.}
\vspace{0.25em}

By Proposition~\ref{prop:graph_properties} (5), there exists a unique fixed vertex $v$. Note that $v$ contains all boundary components of $E_L$. Let $\mathcal{E}(v)=\{T_1,\ldots, T_n\}$ and write $$\G_L-v=\bigsqcup_{i=1}^n\G_i,$$ where $\G_i$ contains the other end of $T_i$ than $v$. By Proposition~\ref{prop:many-edge-cut}, each $\G_i$ is the companionship graph of a nontrivial knot $K_i$. 

Moreover, since each $K_i$ is not the unknot and the torus $T_i$ has to be compressible on at least one side of $S^3-T_i$, the torus $T_i$ is incompressible in $E_{K_i}$. Hence, the edge $T_i$ is directed towards $v$. Then the sublink $L(v)_{\mathcal{E}(v)}$ of $L(v)$ is unlinked by Lemma~\ref{lemma:unlink}. Now the $(n,0)$-amphichiral link $(L,f)$ has the vertex $v$ whose link $L(v)$ satisfies the hypothesis of Lemma~\ref{lemma:elementary-graph}. Therefore, there exists another map $g:(S^3,L)\to (S^3,L)$ such that $(L,g)$ is a strongly negative amphichiral link, which is then slice in $B^4$ by Lemma~\ref{lemma:str} if $|L|>1$, or in the Kawauchi manifold $V$ if $|L|=1$ by \cite{Kaw09, Lev23}.

\vspace{0.25em}
\noindent \textbf{Case 2.} \emph{$f$ fixes at least one edge in $\G_L$.}
\vspace{0.25em}

We prove by induction on the complexity $c(L)$ that there exists some finite integer $n(L)\geq 0$ such that $L$ is slice in $\natural^{n(L)} V$. By cutting along a fixed edge in $\G_L$, we reduce ourselves to proving the statement for certain other negative amphichiral links $L'$ with strictly smaller complexities. Since the number of fixed edges in $\G_{L'}$ is less than the one of $\G_L$, we may regard the base case as \textbf{Case 1}. Hence, assume now that any negative amphichiral link $L'$ such that $c(L') < c(L)$ is slice in $\natural^{n(L')} V$ for some $n(L')$.

\vspace{0.25em}
\noindent $\blacktriangleright$ \textbf{Case 2.a.} \emph{$f$ fixes a coherently directed edge.}
\vspace{0.25em}

Let $T$ be a coherently directed edge. By the edge-cut along $T$, we obtain two subgraphs of $\G_L$. Write $$\G_L - T = \G_1 \sqcup \G_2,$$ where $\G_1$ is the subgraph where $T$ starts from. Let $L_i$ be the link $L(\G_i)$ for $i=1, 2$. Write $L_1 = K_0 \cup \cdots \cup K_{p}$ and $L_2 = J_0 \cup \cdots \cup J_{q}$ respectively, where $p + q = |L|$. Our starting link $L$ is obtained by splicing $L_1$ and $L_2$ along the components $K_0=(L_1)_T$ of $L_1$ and $J_0=(L_2)_T$ of $L_2$, where $J_0$ is unknotted. In other words, $$L=L_1\underset{K_0\sim J_0}{\bowtie} L_2.$$

The components $K_1,\dots, K_p$ of $L_1$ and $J_1,\dots,J_q$ of $L_2$ correspond to the components of $L$. Since $f$ fixes $T$ and its endpoints by Proposition~\ref{prop:graph_properties} (4), $f$ cannot exchange $\G_1$ and $\G_2$. Hence, the action of $f$ restricts on each $\G_i$.  Since $T$ is coherently directed, from the restricted action of $f$, we have that

\begin{itemize}
    \item $f(K_i) = -K_i$ for $i = 0,\ldots, p$,
    \item $f(J_0)=+J_0$,
    \item $f(J_i)=-J_i$ for $i=1,\ldots,q$.
\end{itemize}

\noindent In particular, $L_1$ is again negative amphichiral and $L_2$ is $(q, 1)$-amphichiral.

Since the Gromov norm of a link complement is independent with the directions of edges in the companionship graph, and $\G_1=\G_{L_1}$ and $\#\G_2 = \#\G_{L_2}$ by Proposition~\ref{prop:edge-cut}, we have $c(L)=c(L_1)+c(L_2)$. Moreover, since $c(L_i)\neq 0$ for $i=1,2$, we have $c(L_i)<c(L)$ for $i=1,2$. Therefore, by induction, $L_1$ is slice in $\natural^{n(L_1)} V$ for some $n(L_1)$. By Lemma~\ref{lemma:splice_concordance} (1), $L=L_1\bowtie L_2$ is concordant to $U^{p+1}\bowtie L_2=U^{p}\sqcup (L_2-J_0)$ in $\natural^{n(L_1)}V$. Let $L' = L_2- J_0$, which is again negative amphichiral. By Lemma~\ref{lemma:decreasing_complexity}, we have that $c(L')\leq c(L_2)$, hence $c(L')<c(L)$. Then, by induction on the complexity, $L'$ is slice in $\natural^{n(L')} V$ for some $n(L')$. Therefore, $L$ is in turn slice in $\natural^{n(L_1)+n(L')} V$.

\vspace{0.25em}
\noindent $\blacktriangleright$ \textbf{Case 2.b.} \emph{$f$ fixes an undirected edge.}
\vspace{0.25em}

Let $T$ be an undirected fixed edge. Note that this only occurs when $L$ has more than one component. By definition, $T$ being undirected means that the edge $T$ with endpoints $v$ and $w$ bounds a solid torus in both sides $v$ and $w$. In particular, the components $L(v)_T$ and $L(w)_T$ of $L(v)$ and $L(w)$ are unknots. In this case, by simply set $v$ as the vertex such that $\epsilon_T^v=-1$, write $$\G_L - T = \G_1 \sqcup \G_2,$$ where $\G_1$ contains $v$. In this case, since $T$ is undirected, each $\G_i$ is the companionship graph Proposition~\ref{prop:edge-cut}. Then the same argument in \textbf{Case 2.a} applies.

\vspace{0.25em}
\noindent $\blacktriangleright$ \textbf{Case 2.c.} \emph{every fixed edge is incoherently directed.}
\vspace{0.25em}

Consider the subgraph of $\G$ consisting of the vertices and edges fixed by $f$, and let $\overline{\G}$ be a connected component of     it containing at least one edge, which exists by assumption. Then $\overline{\G}$ is a tree where every edge is directed. If every vertex of $\overline{\G}$ had an exiting edge and an entering edge, then $\overline{\G}$ would be infinite or have a loop. Hence, there exists a vertex $v\in\overline{\G}$ such that every edge in $\overline{\G}$ incident to $v$ is directed to $v$. Thus, every fixed edge in $\G_L$ incident to $v$ is directed to $v$. In particular, they are incoherently directed.

Let $\G'$ be the maximal subtree of $\G$ containing $v$ with no fixed edges, $\mathcal{E}(\G')=\{T_1,\ldots, T_n\}$, and write $$\G- \G'=\bigsqcup_{i=1}^n\G_i,$$ where the subgraphs $\G'$ and $\G_i$ in $\G$ are joined by a single fixed edge $T_i$. Observe that all the subgraphs $\G',\G_1,\ldots,\G_n$ are invariant under $f$ since each $T_i$ is fixed. Since there are no fixed vertex in $\G'$ other than $v$ by Proposition~\ref{prop:graph_properties} (5), every $T_i$ is incident to $v$. Letting the subset $\mathcal{T}$ of $\mathcal{E}(v)$ be the fixed edges, $\mathcal{T}=\mathcal{E}(\G')$.

By Proposition~\ref{prop:edge-cut}, $\G_i$ is the companionship graph of the link $L_i=L(\G_i)$. Note that $c(L_i) < c(L)$. Let $L'$ be the link $L(\G')$. Let $U_i=(L')_{T_i}$ and $K_i= (L_i)_{T_i}$ be the components indexed by $T_i$ of $L'$ and $L_i$, respectively. By Lemma~\ref{lemma:unlink}, we have that $(L')_{\mathcal{T}}=U_1\cup\dots\cup U_n\subset L'$ is an unlink. Then $L$ is a spliced link $$L=(L_1,\dots,L_n)\underset{K_i\sim U_i}{\bowtie} L'.$$

Note that $L'$ is indexed by the disjoint union of $A'=A(\G')\subset A$ and $\mathcal{T}$. Since every $T_i$ is incoherently directed to $v\in \G'$, the link $L'$ is negative amphichiral. Moreover, since $\G'$ has no fixed edges, so does $\G_{L'}$. Thus, we see that $L'$ is strongly negative amphichiral by \textbf{Case 1}.

Recall that the sublink $(L')_{\mathcal{T}}=U_1\cup\dots\cup U_n$ of $L'$ forms an $n$-component unlink $U^n$. Then, by Lemma~\ref{lemma:unknotted_concordance}, the sublink $(L')_{A'}$ is concordant to an unlink $U^p$ in $(S^3-(L')_{\mathcal{T}})\times I$, where $p=|A'|$ is the number of components of $\partial E_L$ contained in $v$. Then, according to Lemma~\ref{lemma:splice_concordance} (2), we have that $L$ is concordant to the split union $\bigsqcup_{i=1}^n (L_i')\sqcup U^p$ in $S^3\times I$, where $L_i'=L_i-K_i$.

Since $L_i'$ is indexed by $A(\G_i)\subset A$ for every $i=1,\ldots, n$, we have that $L_i'$ is negative amphichiral via the appropriate restriction of $f$. Moreover, $c(L_i')\leq c(L_i)$ by Lemma~\ref{lemma:decreasing_complexity} so that $c(L_i')<c(L)$. Then, by induction, $L_i'$ is slice in $\natural^{n(L_i')}V$ for some $n(L_i')$. Therefore, the link $L$ is slice in $\natural^{n(L)}V$ where $n(L)= n(L_1')+\dots+n(L_n')$.
\end{proof}


\section{Proofs of Theorems~\ref{thm:SNACK} and \ref{thm:SNACK-concordance}}
\label{sec:5}

In this section, we prove Theorem~\ref{thm:SNACK}, which extends the work of Kawauchi \cite{Kaw79} on hyperbolic negative amphichiral knots, and Corollary~\ref{cor:fibered}, which answers the question on Miyazaki knots asked by \cite{KW18}. Furthermore, we prove Theorem~\ref{thm:SNACK-concordance}, which gives a partial answer to Question~\ref{question:con-to-SNACK}, and hence to Question~\ref{question:single-Kaw}.

First of all, we motivate the notion of \text{totally coherent} JSJ structure of a negative amphichiral knot from the proof of Theorem~\ref{thm:link}, and recall its definition to prove Theorem~\ref{thm:SNACK}. Recall that every hyperbolic negative amphichiral knot is strongly negative amphichiral \cite{Kaw79}. In \textbf{Case 1} of Theorem~\ref{thm:link}, we have proved that every negative amphichiral knot whose companionship graph has no fixed edge is strongly negative amphichiral. Since any hyperbolic knot has only a single JSJ piece, it falls under \textbf{Case 1}.

Recall that in \textbf{Case 2.c} we have further considered a maximal subtree $\G'$ without fixed edges of $\G_K$ containing a certain vertex $v$, and such $v$ can be chosen to be the root when the given link is a knot. \textbf{Case~1} is the case when $\G_K = \G'$. The notion of the \textit{maximal coherent subtree} $\G_{max}$ generalizes $\G'$: it is maximal among the subtrees $\G$ containing $v$ and satisfying that every fixed vertex of $\G$ is coherent. If $\G_K$ has no fixed edge, then $\G_{max}=\G'$ is just $\G_K$ itself. Recall that we call a negative amphichiral knot $K$ has a \textit{totally coherent} JSJ structure if $\G_{max} = \G_K$. Now we restate Theorem~\ref{thm:SNACK}:

\begin{repintrothm}{thm:SNACK}
    Let $(K, f)$ be a negative amphichiral knot. If the maximal coherent subtree $\G_{max}$ is the same as $\G_K$, then there exists an orientation-reversing involution $g$ on $S^3$ such that $(K,g)$ is strongly negative amphichiral.
\end{repintrothm}

\begin{remark}
  Note that the provided involution $g$ for $K$ in Theorem~\ref{thm:SNACK} may not be isotopic to $f$ on $E_K$.
\end{remark}

Before proving Theorem~\ref{thm:SNACK}, we first provide a lemma that will be crucially used to prove not only Theorem~\ref{thm:SNACK}, but also Theorem~\ref{thm:SNACK-concordance}.

\begin{lem}\label{lemma:coherent_root}
Let $(K, f)$ be a negative amphichiral knot. If the root $v$ of $\G_K$ is coherent, then the corresponding link $L$ to the maximal coherent subtree $\G_{max}$ of $\G_K$ is a strongly $(1, n)$-amphichiral link of the form: $$L = J_0 \cup U_1 \cup \cdots \cup U_n,$$ where $J_0$ is the component corresponding to $\partial E_K$ and the sublink $U_1 \cup \cdots \cup U_n$ is unlinked.
\end{lem}

\begin{proof}
Since $v\in \G_{max}$ by assumption, $\G_{max}\neq \varnothing$. Let  $\mathcal{E}(\G_{max})=\{T_1,\ldots, T_n\}$. We will first check that $(L, F_{\G_{max}})$ is a $(1,n)$-amphichiral link of the desired form. Then, by performing edge-cuts along all fixed edges in $\G_{max}$, we will see that the link for each piece is also amphichiral. Then we will obtain amphichiral involution for each piece, and finally we will glue them to get the $(1,n)$-amphichiral involution of $L$.

The link $L$ is indexed by $\partial E_K$ and $\mathcal{E}(\G_{max})$. Observe that every $\mathcal{E}(\G_{max})$ is directed to $\G_{max}$ so that $L_{\mathcal{E}(\G_{max})}$ is unlinked by Lemma~\ref{lemma:unlink}. Let $w$ be a vertex $\G_{max}\cap T_i$. Then $w$ is either not fixed or coherent. If $w$ were not fixed, then $\G_{max}$ would contain $T_i$, so $w$ should be fixed, and hence coherent. Since every $T_i$ is incident to a coherent vertex in $\G_{max}$, $T_i$ is coherently directed or not fixed. However, if some $T_i$ were not fixed, then it would be contained in $\G_{max}$ so that $T_i\notin \mathcal{E}(\G_{max})$. Thus, every $T_i\in \mathcal{E}(\G_{max})$ is coherently directed. Therefore, $L$ is a $(1,n)$-amphichiral link of the desired form, where $J_0=L_{\partial E_K}$ and $U_1\cup \cdots \cup U_n = L_{\mathcal{E}(\G_{max})}$.

Let $\mathcal{T}$ be the set of fixed edges in $\G_{max}$, and write $$\G_{max}-\mathcal{T}=\G_0\sqcup\G_1\sqcup\cdots\sqcup\G_l,$$ where $\G_0$ contains the root $v$. Since each $T\in \mathcal{T}$ is fixed, $f(\G_i)=\G_i$ for each $i=0,\ldots, l$, and each $\G_i$ has no fixed edges. Then, by Proposition~\ref{prop:graph_properties} (5), each $\G_i$ has a unique fixed vertex $v_i$, and note that $v_0=v$.

Let $L_i$ be the corresponding link $L(\G_i)$ and $F_{\G_i}$ be the map on $(S^3, L_i)$ induced from $f$. We will first verify that each $(L_i, F_{\G_i})$ is $(1, n_i)$-amphichiral, where $n_i$ is $|\mathcal{E}(\G_i)| - 1$ for $i\ge 1$ and $|\mathcal{E}(\G_0)|$ for $i=0$. Then we will check the condition of Lemma~\ref{lemma:elementary-graph} to obtain a $(1,n_i)$-amphichiral involution on $L_i$.

For $i\ge 1$, since $v\notin \G_i$, $L_i$ is indexed by edges $\mathcal{E}(\G_i)$ in $\G_K$. Since every such edge is fixed either in $\G_{max}$ or in $\mathcal{E}(\G_{max})$, it is always coherent. Let $A_i^-$ and $A_i^+$ be the exiting and entering subsets of $\mathcal{E}(\G_i)$, respectively. Since $v_i$ is the only fixed vertex in $\G_i$, every edge of $\mathcal{E}(\G_i)$ is incident to $v_i$ so that $|A_i^-|=1$. Thus, $L_i$ is $(1, n_i)$-amphichiral.

$L_0$ is indexed by $\mathcal{E}(\G_0)$ and $\partial E_K$. In this case, every edge in $\mathcal{E}(\G_0)$ is coherent by a similar argument, but only entering $\G_0$. Instead of the exiting edge, the component indexed by $\partial E_K$ forms the negative amphichiral component so that $L_0$ is also $(1, n_0)$-amphichiral. For the time being, we also regard $\partial E_K$ as an exiting edge from $v$, so we let $A_0^-=\{\partial E_K\}$ and $A_0^+ = \mathcal{E}(\G_0)$ for convenience.

Now we consider $L(v_i)$. The set $\mathcal{E}(v_i)$ consists of $A_i^\pm$ and the set $A_i^e$ of edges not fixed. Since every edge except the one in $A_i^-$ is directed to $v_i$, by Lemma~\ref{lemma:unlink}, the sublink $L(v_i)_{A_i^+\cup A_i^e}$ is unlinked. Therefore, by applying Lemma~\ref{lemma:elementary-graph}, we obtain a $(1, n_i)$-amphichiral involution $g_i:(S^3,L_i)\to (S^3,L_i)$. In particular, such $g_i$ can be taken to satisfy $\Fix(g_i)\cong S^0$, which acts in the same fashion as $F_{\G_i}$ on the components of $L_i$. 

Let $T$ be a fixed edge in $\G_{max}$ directed from $v_i$ to $v_j$, namely the torus $T=\partial v_i\cap \partial v_j = \partial E_{L_i}\cap\partial E_{L_j}$.  Note that one of the components $(L_i)_T$ and $(L_j)_T$ is strongly $(-)$-amphichiral and the other is strongly $(+)$-amphichiral, with the involutions $g_i$ and $g_j$. Now we glue all $g_i$ along such $T$ and obtain the $(1,n)$-amphichiral involution of $(S^3, L)$.

Since our involution $g_i$ satisfies $\Fix(g_i)\cong S^0$ and $\Fix(g_i)$ lies in the component $(L_i)_{A_i^-}$, there are no fixed points in the complement $E_{L_i}$. In particular, their restrictions to $T$ are fixed-point-free involutions. Since $(\pm)$-amphichiral involutions on a torus without fixed points are unique up to conjugate with a map isotopic to the identity by \cite{Har80},\footnote{In \cite{Har80}, involutions on torus are characterized up to conjugate with a map isotopic to the identity, namely up to \emph{special equivalence}. In his notations, $(-)$-amphichiral involution is $R_3$, and $(+)$-amphichiral involution is $R_4$.} and they differ by the meridian-longitude exchanging map, we can glue all the maps $g_i$ along such tori to get an $(1, n)$-amphichiral involution $g:(S^3,L)\to (S^3,L)$, acting on the components of $L$ in the same way as $F_{\G_{max}}$.
\end{proof}

The knot $J_0$ given in Lemma~\ref{lemma:coherent_root} plays the role of the strongly negative amphichiral knot in Theorem~\ref{thm:SNACK}. Now we are ready to prove Theorem~\ref{thm:SNACK}.

\begin{proof}[Proof of Theorem~\ref{thm:SNACK}]

Since $\G_K=\G_{max}$, the corresponding link $L(\G_{max})$ is $K$ itself. Moreover, since the root $v$ of $\G_K$ is coherent by definition, by Lemma~\ref{lemma:coherent_root}, $K$ is strongly negative amphichiral. More precisely, all the arguments in the proof of Lemma~\ref{lemma:coherent_root} work when $\G_{max}=\G_K$ so that $n=0$ and $L=K$.
\end{proof}

Recall that Kim and Wu \cite{KW18} asked if every Miyazaki knot, namely, a fibered negative amphichiral knot with irreducible Alexander polynomial, is strongly negative amphichiral. Now we prove that every \textit{fibered} negative amphichiral knot admits a totally coherent JSJ structure, which is subtly indicated by Kim and Wu \cite{KW18}. Corollary~\ref{cor:fibered} is a direct consequence of the following lemma.

\begin{lem}\label{lemma:fibered}
A fibered negative amphichiral knot $(K,f)$ admits a totally coherent JSJ structure.
\end{lem}
\begin{proof}
In order to prove that $\G_K = \G_{max}$, i.e., every fixed vertex is coherent, we show that every fixed edge is coherently directed. Let $T$ be a fixed edge directed from $w$ to $v$. We prove that $\epsilon_{T,v}=+1$.
Write $$\G_K-T=\G_1\sqcup \G_2,$$ where $\G_1$ contains $v$ and $\G_2$ contains $w$. Note that $\G_1$ contains the root of $\G_K$.

Let $L$ be the link $L(\G_1)$. Since $M(\G_1)$ has two boundary components from $\partial E_K$ and $T$, $L$ has two components, and we write $L=P\cup U$, where $P$ is the component from $\partial E_K$ and $U=L_T$. For example, see Figure~\ref{fig:lemma5.3}. Note that $U$ is unknotted since $T$ is directed to $v\in \G_1$. Moreover, $\G_2$ is the companionship graph $\G_J$ for some nontrivial knot $J$ by Proposition~\ref{prop:many-edge-cut}. Then, the spliced knot $K=J\bowtie L$ can be described as a satellite of the companion $J$ and the pattern given by $P\subset S^3-U=S^1\times D^2$, i.e., $K=P(J)$.

Since $K$ is fibered, according to \cite[Theorem~2.1]{HMS08}, we have that the winding number of the pattern, i.e., the linking number between $P$ and $U$, is non-zero. Now if $\epsilon_{T,v}=-1$ then $(L,F_{\G_1})$ would be a negative amphichiral link, and hence $\lk(P,U)=0$. Therefore $\epsilon_{T,v}=+1$ so that $T$ is coherently directed. Since this holds for every fixed edge $T$, we deduce that $(K,f)$ is admits a totally coherent JSJ structure.
\end{proof}

\begin{figure}[htbp]
\centering
\includegraphics[width=0.7\columnwidth]{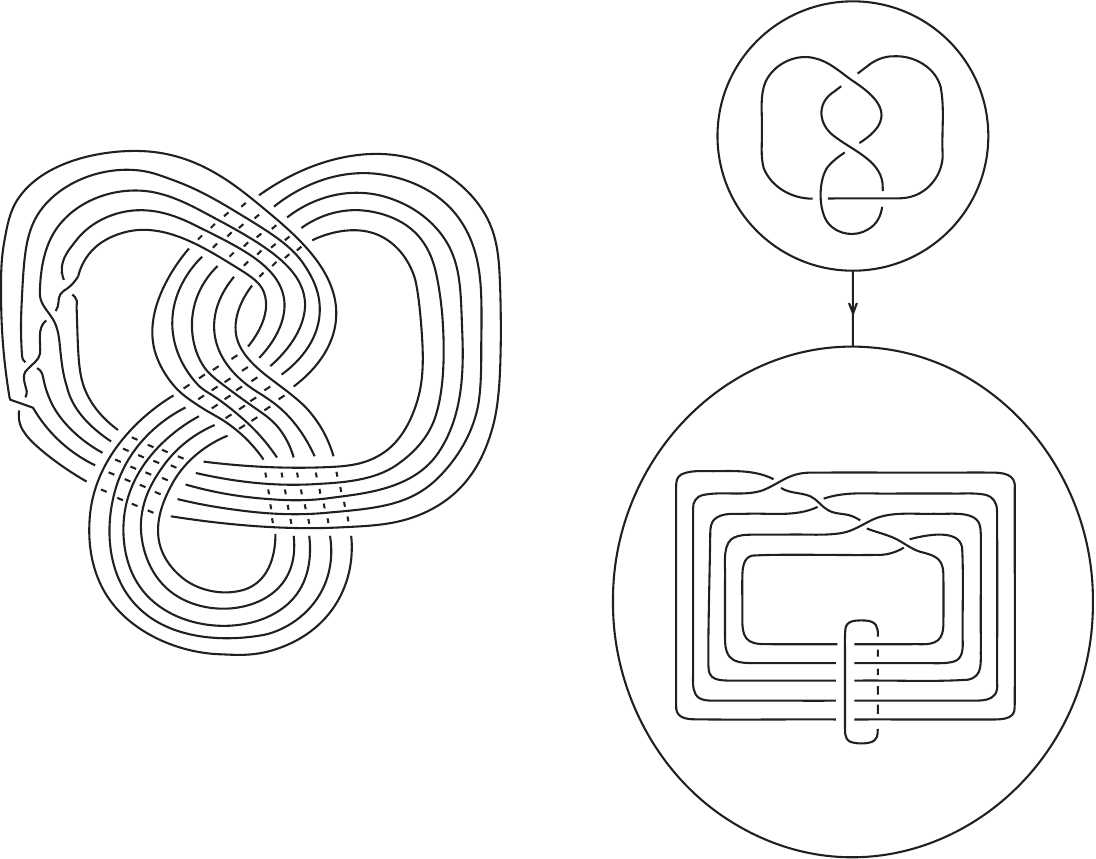}
\caption{Left: A fibered negative amphichiral knot $K = P(J)$. Right: the companionship graph $\G_K$ of $K$.}
\label{fig:fibered}
\end{figure} 

An example of a satellite Miyazaki knot $K = P(J)$ with its companionship graph $\G_K$ is depicted in Figure~\ref{fig:fibered} and \cite[Example~3]{KW18}. Here, $J$ is the figure-eight knot $4_1$, and $P(U)$ is the unknotted pattern given by the Turk's head knot $P(U) = Th(5,1)$ (see \cite{DPS24} for a survey on Turk's head knots). By following the recipe \cite[Proposition~4.1]{KW18}, the ingredients $P(U)=Th(p,1)$ with odd $p$ and $J=4_1$ will provide more Miyazaki knots. 

Furthermore, in our purpose to obtain a fibered negative amphichiral knot, we do not require $P(U)$ to be unknotted or the Alexander polynomial of $P(J)$ to be irreducible. Since every Turk's head knot $Th(p,q)$ is fibered as a braid closure in $S^1\times D^2$ by \cite{Sta78} and it is known that $P(J)$ is fibered if and only if $P(U)$ in $S^1\times D^2$ and $J$ are fibered by \cite{HMS08}, we can just take $P(U)=Th(p, q)$ with odd $p$ and $\gcd(p,q)=1$, and $J$ as any fibered negative amphichiral knot, to produce more fibered satellite negative amphichiral knots $P(J)$.

Now we move on to Theorem~\ref{thm:SNACK-concordance}. We will motivate the notions of \textit{properly incoherent} JSJ structure and \textit{incoherent} root in a top-down perspective. Then we provide a lemma concerning the case where $\G_K$ has the incoherent root, and prove Theorem~\ref{thm:SNACK-concordance}.

Recall that we can always obtain the splice form of $K$ as $K = (K_1,\ldots, K_n)\bowtie L$ with $L=L(\G_{max})$ whenever $\G_{max}$ is nonempty. When the root is coherent, equivalently $\G_{max}\neq \varnothing$, $L$ is a strongly $(1,n)$-amphichiral link of the form in Lemma~\ref{lemma:coherent_root}. Thus, if every $K_i$ is slice, then by Lemma~\ref{lemma:splice_concordance} (2), $K$ is concordant to the strongly negative amphichiral knot $J_0$, as desired. Note that each $K_i$ is a negative amphichiral knot with the root which is not coherent, and each root has at least one incoherently directed edge. Thus, we seek a condition for such a negative amphichiral knot to be slice.

Let $J$ be a negative amphichiral knot such that there exists at least one incoherently directed edge in $\G_J$ incident to the root $v$. We similarly proceed with \textbf{Case 2.c} of proof of Theorem~\ref{thm:link} to $\G_J$. Namely, we remove all fixed edges in $\G_J$ to get the splice form $J = (J_1,\ldots,J_m)\bowtie L$ where $L$ is $(p,q)$-amphichiral with $p\ge 2$. More precisely, $p$ is the number of incoherently directed edges incident to the root $v$, plus one from $\partial E_J$, and $q$ is the number of coherently directed edges in $\mathcal{E}(v)$. To apply Lemma~\ref{lemma:unknotted_concordance} to conclude that $J$ is slice, we desire that $L$ is a strongly negative amphichiral link of the form in Lemma~\ref{lemma:str}. Then we simply require that all fixed edges incident to $v$ be incoherently directed to use Lemma~\ref{lemma:elementary-graph}.
 
Wrapping up all the reverse engineering above, we recall that $v$ is \textit{incoherent} if there exists at least one fixed edge incident to $v$, and every such edge is incoherently directed. Recall also that $K$ has a \textit{properly incoherent} JSJ structure if $\G_K-\G_{max}= \sqcup_{i=1}^n \G_i$ and each $\G_i$ has the incoherent root. We don't have to assume the root of properly incoherent $\G_K$ is coherent because otherwise $\G_{max}=\varnothing$ so that $\G_K$ itself has the incoherent root by definition as Proposition~\ref{prop:coherent-properties} (3) and (4).
Now we are ready to prove Theorem~\ref{thm:SNACK-concordance}.

\begin{repintrothm}{thm:SNACK-concordance}
    Let $(K, f)$ be a negative amphichiral knot. For the maximal coherent subtree $\G_{max}$ of $\G_K$, write $$\G_K - \G_{max} = \bigsqcup_{i=1}^n\G_i.$$ If the root of each $\G_i$ is incoherent, then there exists a strongly negative amphichiral knot $J$ such that $K$ and $J$ are concordant.
\end{repintrothm}

\noindent We first provide the following lemma as desired:

\begin{lem}\label{lemma:incoherent_root}
Let $(K,f)$ be a negative amphichiral knot. If $\G_K$ has the incoherent root, then $K$ is slice.
\end{lem}

\begin{proof}
Let $\mathcal{T}=\{T_1,\ldots, T_n\}\subset \mathcal{E}(v)$ be the set of the fixed edges incident to the root $v$. Write $$\G_K-\mathcal{T}=\G_0\sqcup \G_1\sqcup \dots\sqcup\G_n,$$ where $\G_0$ is the connected component containing $v$, see Figure~\ref{fig:lemma5.2}. Let $L$ be the link $L(\G_0)$. Since $\G_0$ contains the root $v$, each of $\G_i$ is the companionship graph of a nontrivial knot $J_i$ by Proposition~\ref{prop:many-edge-cut}. Since $L$ is indexed by $\partial E_K$ and $\mathcal{T}=\{T_1,\ldots, T_n\}$, and each $T_i$ is incoherently directed, $L$ is an $(n+1)$-component negative amphichiral link, and $L_{\mathcal{T}}$ is unlinked by Lemma \ref{lemma:unlink}. For simplicity, we write the components $L_{T_i}=U_i$, and $L-L_{\mathcal{T}}=J$. Then the knot $K$ is the resulting spliced knot: $$K = (J_1,\dots,J_n)\underset{J_i\sim U_i}{\bowtie} L.$$

  \begin{figure}[htbp]
\centering
\includesvg[width=0.55\columnwidth]{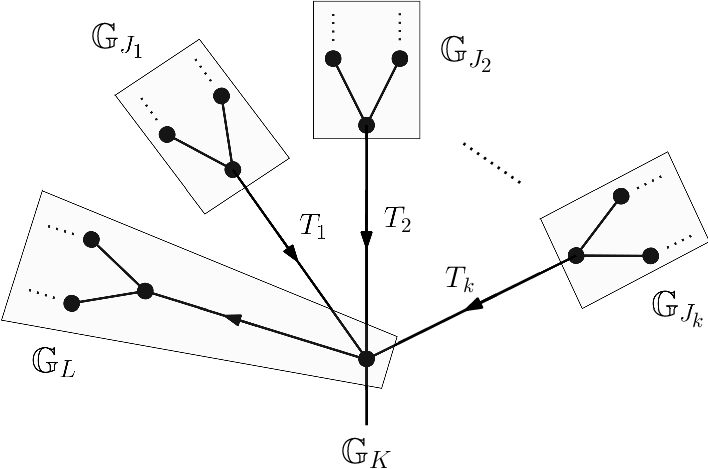}
\caption{The edge-cuts of $\G_K$ along the fixed edges $T_i$ incident to the root $v$ for $i=1,\ldots,n$.}
\label{fig:lemma5.2}
\end{figure} 

For the proof, it is enough to show that $L$ is strongly negative amphichiral. The reason is, by Lemma~\ref{lemma:unknotted_concordance}, if $L$ is strongly negative amphichiral, then $J$ is concordant to a local unlink in $(S^3-L_{\mathcal{T}})\times I$. From this concordance, applying Lemma~\ref{lemma:splice_concordance} (2), $K$ is concordant to the splice of $(J_1,\cdots, J_n)$ and the $(n+1)$-component unlink, which is the unknot.

Now we prove that $L$ is strongly negative amphichiral by applying Lemma~\ref{lemma:elementary-graph}. Let $A$ be the index set of $L$, which consists of $\partial E_K$ and $\mathcal{T}$, and consider the companionship graph $\G_L$ obtained from $\G_0$ by undirecting $DC(T_i)$ by Proposition~\ref{prop:edge-cut}. Since any edges in $\G_0$ incident to $v$ are not fixed, the same holds for such edges in $\G_L$. Then, by Proposition~\ref{prop:graph_properties} (5), there does not exist a fixed vertex in $\G_L$ other than $v$, so $L(v)$ is indexed by the disjoint union of $A$ from $\partial E_L = \partial E_K \cup \mathcal{T}$, and $\mathcal{E}(v)\cap \G_0$ from the incident edges to $v$ in $\G_L$.

Since $\G_K$ is a rooted tree, every incident edge to $v$ in $\G_0$ is directed to $v$. Then by applying Lemma~\ref{lemma:unlink} to $\G_K$ and $v$, the sublink $L(v)_{\mathcal{E}(v)\cap \G_0}$ of $L(v)$ is unlinked. Thus, the conditions for $L$ and $\G_L$ of Lemma~\ref{lemma:elementary-graph} hold. Therefore, we can conclude that $L$ is strongly negative amphichiral, which completes the proof.
\end{proof}

\begin{proof}[Proof of Theorem~\ref{thm:SNACK-concordance}]

Let $v$ be the root of the companionship graph $\G_K$. Recall from Proposition~\ref{prop:coherent-properties} (3) that if $v$ is not coherent, then $\G_K =\varnothing$ so $\G_K$ is properly incoherent if and only if $\G_K$ has the incoherent root by Proposition~\ref{prop:coherent-properties} (4). In this case, by Lemma~\ref{lemma:incoherent_root}, $K$ is concordant to the unknot, which is strongly negative amphichiral.

Hence, from now on, assume that $v$ is coherent. Let $L$ be the link $L(\G_{max})$ corresponding to the maximal coherent subtree $\G_{max}$ of $\G_K$. Let $T_i$ be the edge joining $\G_{max}$ and $\G_i$. Note that $L$ is indexed by $\partial E_K$ and $\mathcal{T}=\{T_1,\ldots, T_n\}$. By Lemma~\ref{lemma:coherent_root},  the component $J_0=L_{\partial E_K}$ of $L$ is strongly negative amphichiral and $L_{\mathcal{T}}$ is unlinked.

Let $K_i$ be the nontrivial knot such that $\G_{K_i}=\G_i$ obtained from Proposition~\ref{prop:many-edge-cut}. Let $U_i$ be the component $L_{T_i}$. Note that the knot $K$ is the resulting knot of the splice form:
$$K = (K_1,\dots,K_n)\underset{K_i\sim U_i}{\bowtie} L.$$
Since each $\G_{K_i}=\G_i$ has the incoherent root, by Lemma~\ref{lemma:incoherent_root}, there exist concordances in $S^3\times I$ between $K_i$ and the unknot for any $i$. By Lemma~\ref{lemma:splice_concordance} (1), we conclude that $K$ is concordant to $U^n\bowtie L=J_0$, which is a strongly negative amphichiral knot.
\end{proof}


\section{Potential candidates for Question~\ref{question:single-Kaw}}
\label{sec:6}

Back to Question~\ref{question:con-to-SNACK}, while we give a partial answer positively in Theorem~\ref{thm:SNACK-concordance}, there are negative amphichiral knots whose companionship graphs are not properly incoherent. These kinds of knots form potential candidates for Question~\ref{question:single-Kaw}, and hence potential counterexamples to  Question~\ref{question:con-to-SNACK}. We say a $(\pm)$-amphichiral knot \emph{weakly $(\pm)$-amphichiral} if it is not strongly $(\pm)$-amphichiral, respectively. We close this article by presenting several potential candidates for Question \ref{question:single-Kaw}, namely, weakly negative amphichiral knots that may not be slice in the Kawauchi manifold $V$.

Recall that the \emph{Fox--Milnor condition} \cite{FM66} states that the Alexander polynomial $\Delta_K (t)$ of a slice knot $K$ factors as $\Delta_K (t) = f(t) f(t^{-1})$ for some polynomial $f$ with integer coefficients.

\begin{prop}\label{prop:cand1}
    If $K_1$ is negative amphichiral, and $K_2$ and $K_3$ are positive amphichiral, then the knot $K$ in Figure~\ref{fig:cand1} is a negative amphichiral knot whose $\G_K$ is not properly incoherent. Moreover,
    \begin{itemize}
        \item If at least one $K_i$ is weakly amphichiral, then $K$ is weakly negative amphichiral.
        \item If the Alexander polynomial of $K_1$ fails to satisfy the Fox--Milnor condition, then $K$ is not slice.
    \end{itemize}

    \begin{figure}[h]
  \includesvg[scale = 0.7]{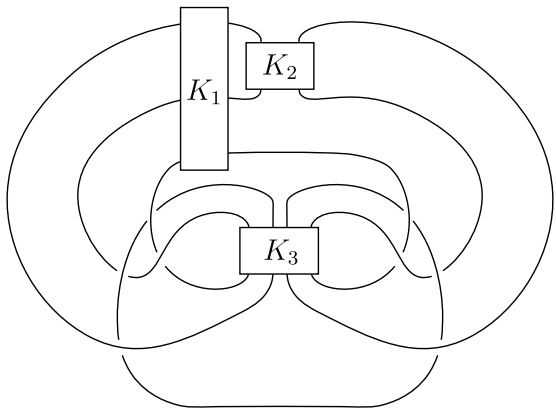}
  \caption{A negative amphichiral knot $K$ whose JSJ structure is not properly incoherent.}
\label{fig:cand1}
\end{figure}
\end{prop}

\begin{proof}
Consider the companionship graph $\G_K$ of $K$ as illustrated in Figure~\ref{fig:cand1-grp} with the root $v$. Each subgraph $\G_i$ for $i=1,2,3$ is the companionship graph $\G_{K_i}$ of $K_i$. Let $T_i$ be the edge from the root of $\G_i$ to $v.$ Let $L$ be the link $L(v)$ labeling $v$ indexed by $\{*,1, 2, 3\}$ as in Figure~\ref{fig:cand1-grp}. Then the resulting $K$ has a splice form as follows: $$K = (K_1,K_2,K_3)\underset{K_i\sim L_i}{\bowtie}L.$$

\begin{figure}[h]
\includesvg[scale=0.6]{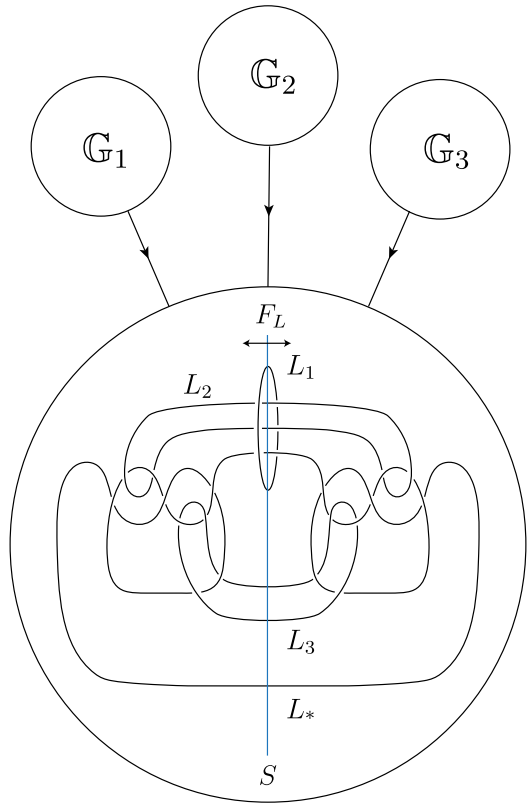}
  \caption{The companionship graph $\G_K$ of the knot $K$ in Figure~\ref{fig:cand1}. The reflection $F_L$ along the blue-colored $2$-sphere $S$ is the $(3,1)$-amphichiral involution of $L$. Note that the component $L_1$ is contained in the sphere $S$.}
\label{fig:cand1-grp}
\end{figure}

The link $L$ has an obvious $(3,1)$-amphichiral involution $F_L$ on $(S^3, L)$ by the reflection along the sphere $S$ such that
    \begin{itemize}
        \item $F(L_*) = -L_*$,
        \item $F(L_1) = +L_1$,
        \item $F(L_2) = -L_2$,
        \item $F(L_3) = -L_3$.
    \end{itemize}
Moreover, using the functions \verb|verify_hyperbolicity| and \verb|symmetry_group| of the \verb|Manifold| class in SnapPy \cite{snappy}, we can check that $L$ is hyperbolic and that $F_L$ is the unique symmetry on $(S^3, L)$. Thus, if $K$ admits a negative amphichiral involution $F$, then its restriction $F_v$ on $v$ should be the same as $F_L$.

Suppose that such an $F$ exists. Since $F_L$ sends $L_2$ and $L_3$ to their reverses, the restriction of $F$ on the subgraphs $\G_2$ and $\G_3$ induces positive amphichiral involutions $g_2$ and $g_3$ on $K_2$ and $K_3$, respectively. However, when $K_i$ for $i=2$ or $3$ is not strongly positive amphichiral, such a positive amphichiral involution for $K_i$ does not exist. Thus, $K$ cannot have such a negative amphichiral involution $F$. A similar argument works for the case when $K_1$ is weakly negative amphichiral.

However, when $K_2$ and $K_3$ admit weakly positive amphichiral maps $g_2$ and $g_3$, we can isotope $g_2$ and $g_3$ to make agree with $F_L$ on the common boundary components, namely, the edges $T_2$ and $T_3$. Similarly, by gluing the negative amphichiral involution of $K_1$ on $\G_1$ along $T_1$ via isotopy with $F_L$, we get a negative amphichiral map of $K$.

One can observe that the root $v$ is neither coherent nor incoherent since $T_1$ is coherently directed but $T_2$ and $T_3$ are incoherently oriented. Thus, the maximal coherent subtree of $\G_K$ is empty, which implies that $\G_K$ is properly incoherent if and only if the root of $\G_K$ is incoherent by Proposition~\ref{prop:coherent-properties} (3) and (4). Therefore, $\G_K$ is not properly incoherent.

On the other hand, the component $L_*$ is the knot $T_{2,3}\# \overline{T_{2, 3}}$. Since the Alexander polynomial of $K$ is the same as the one of $T_{2,3} \# \overline{T_{2,3}} \# K_1$ as \cite[Theorem~1]{Bud06}, if the Alexander polynomial of $K_1$ fails to satisfy the Fox--Milnor condition, then so does that of $K$, implying that $K$ is not slice.
\end{proof}

Since $\G_K$ is not properly incoherent, Theorem~\ref{thm:SNACK-concordance} cannot be applied. Moreover, there is no chance to use Lemma~\ref{lemma:splice_concordance} since the link $L$ is clearly not slice. Thus, we do not know if $K$ is concordant to a strongly negative amphichiral knot.

Following the proof of Theorem \ref{thm:link}, one can obtain an upper bound for $\Kaw(K)$ in terms of $\Kaw(K_i)$ as follows. Since $T_1$ is coherently directed, the edge-cut along $T_1$ gives the splice form
$$K = K_1\underset{K_1\sim L_1}{\bowtie}((K_2,K_3)\underset{K_i\sim L_i}{\bowtie} L).$$ Let $K'$ be the knot in the right-hand side of the first bow tie notation. Then by Lemma \ref{lemma:splice_concordance}, $K$ is concordant to $K'$ in $\natural^{\Kaw(K_1)} V$. Then now the root of $\G_{K'}$ is incoherent so that $K'$ is slice by Lemma \ref{lemma:incoherent_root}. Therefore, $$\Kaw(K)\le \Kaw(K_1).$$

We also provide a construction of weakly negative amphichiral knots $K$ that are \emph{topologically} slice. This relies on the fact that a knot of the trivial Alexander polynomial is topologically slice \cite{Fre82, FQ90}. On the other hand, we do not know whether $K$ is smoothly slice or not. As in Proposition~\ref{prop:cand1}, $\G_K$ is not properly incoherent. Since the Borromean ring $L$ is not slice, it is also impossible to apply Lemma~\ref{lemma:splice_concordance}. One can easily obtain $\Kaw(K)\le \Kaw(K_1)$.

\begin{prop}\label{prop:cand2}
    If $K_1$ is negative amphichiral and $K_2$ is positive amphichiral, then the knot $K$ in Figure~\ref{fig:cand2} is negative amphichiral whose $\G_K$ is not properly incoherent. Moreover, if $K_2$ is weakly positive amphichiral, then $K$ is weakly negative amphichiral.

    \begin{figure}[h]
  \includesvg[scale=0.8]{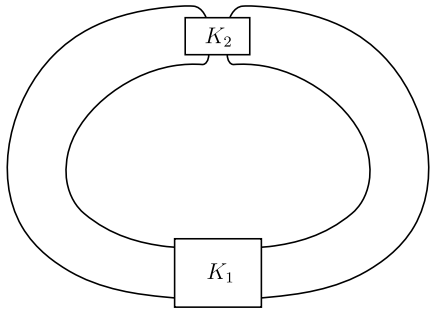}
  \caption{A weakly negative amphichiral knot $K$ which is topologically slice. $K$ does not have a properly incoherent JSJ structure.}
\label{fig:cand2}
\end{figure}

\end{prop}

\begin{proof}

Consider the companionship graph  $\G_K$ of $K$ with the root $v$ illustrated in Figure~\ref{fig:cand2-grp}. Each subgraph $\G_i$ for $i=1, 2$ is the companionship graph $\G_{K_i}$ of $K_i$. Let $T_i$ be the edge from the root of $\G_i$ to $v$. Let $L$ be the corresponding link $L(v)$ to $v$ indexed by $\{*, 1, 2\}$ as in Figure~\ref{fig:cand2-grp}. Then the resulting knot $K$ has a splice form as follows: $$K=(K_1,K_2)\underset{K_i\sim L_i}{\bowtie} L.$$

\begin{figure}[h]    \includesvg[scale=0.6]{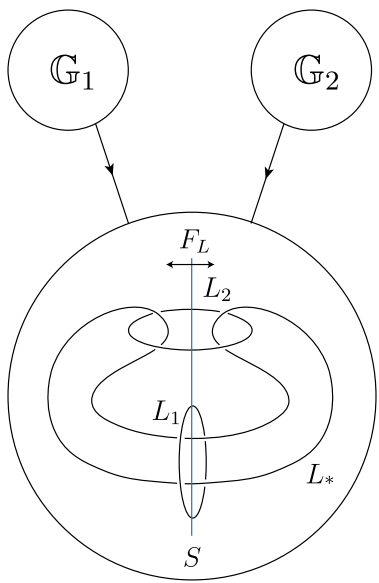}
  \caption{The companionship graph $\G_K$ of the knot $K$ in Figure \ref{fig:cand2}. The reflection $F_L$ along the blue-colored $2$-sphere $S$ is the $(2, 1)$-amphichiral involution on the Borromean ring $L$. Note that the component $L_1$ is contained in the sphere $S$.}
\label{fig:cand2-grp}
\end{figure}

Note that the link $L$ is the Borromean ring, which is hyperbolic. Observe that $L$ has an obvious $(2,1)$-amphichiral involution $F_L$ on $(S^3, L)$ by the reflection along the sphere $S$ such that
\begin{itemize}
\item $F(L_*) = -L_*$,
\item $F(L_1) = +L_1$,
\item $F(L_2) = -L_2$.
\end{itemize}
Moreover, it follows from SnapPy \cite{snappy} that $F_L$ is the unique symmetry on $(S^3, L)$ satisfying the properties above. Thus, if $K$ admits a negative amphichiral involution $F$, then its restriction $F_v$ on $v$ should be the same as $F_L$.

Suppose that such an $F$ exists. Since $F_L$ sends $L_2$ to its reverses, the restriction of $F$ on $\G_2$ induces a positive amphichiral involution $g$ on $K_2$. However, when $K_2$ is not strongly positive amphichiral, then such an involution does not exist. Thus, $K$ cannot have such a negative amphichiral involution $F$.

However, when there exists a positive amphichiral map $g$ which is not necessarily of order $2$, we can isotope $g$ to make it agree with $F_L$ on the common boundary component $T_2$. By gluing the negative amphichiral involution of $K_1$ on $\G_1$ along $T_1$ via isotopy with $F_L$, we get a negative amphichiral map of $K$.

Since $T_1$ is coherently directed and $T_2$ is incoherently directed, the root $v$ is neither coherent nor incoherent. Thus, by the same reason in Proposition~\ref{prop:cand1}, $\G_K$ is not properly incoherent.
\end{proof}

\bibliography{references}
\bibliographystyle{amsalpha}

\end{document}